\documentclass[11pt,a4paper]{article}
\usepackage{amssymb,amsmath,amsfonts,mathrsfs,bm,enumerate,comment}
\allowdisplaybreaks[1]
\numberwithin{equation}{section}

\usepackage[colorlinks=true, pdfstartview=FitV, linkcolor=blue, citecolor=blue, urlcolor=blue,pagebackref=false]{hyperref}

\usepackage{tikz}
\usepackage{float}
\usepackage[font=footnotesize]{caption}

\usepackage{times,theorem,latexsym,color,comment}

\newcommand{\BOX}{\ensuremath\Box}

\newtheorem{theorem}{Theorem }[section]

{\theorembodyfont{\rmfamily}}
{\theorembodyfont{\rmfamily}}
{\theorembodyfont{\rmfamily}}
\newtheorem{lemma}[theorem]{Lemma}
\newtheorem{proposition}[theorem]{Proposition}
{\theorembodyfont{\rmfamily}\newtheorem{remark}[theorem]{Remark}}
{\theorembodyfont{\rmfamily}}

\newcommand{\R}{\mathbb{R}}
\newcommand{\C}{\mathbb{C}}

\newcommand{\dd}{\,{\rm d}}

\newcommand{\opdiv}{\operatorname{div}}

\newcommand{\oparg}{\operatorname{arg}}

\newcommand{\eps}{\varepsilon}
\newcommand{\ii}{\mathrm i}

\newcommand{\overbar}[1]{\mkern 1.5mu\overline{\mkern-1.5mu#1\mkern-1.5mu}\mkern 1.5mu}

\def\XXint#1#2#3{{\setbox0=\hbox{$#1{#2#3}{\int}$}
		\vcenter{\hbox{$#2#3$}}\kern-.5\wd0}}

\newenvironment{proof}{{\vskip\baselineskip\noindent\textbf{Proof:}}}
{\hspace*{.1pt}\hspace*{\fill}\BOX\vskip\baselineskip}

\newenvironment{proofx}[1]
{\vskip\baselineskip\noindent\textbf{Proof of {#1}:}}
{\hspace*{.1pt}\hspace*{\fill}\BOX\vskip\baselineskip}

\begin{document}

\title{An $L^2$-quantitative global approximation \\
for the Stokes initial-boundary value problem}

\author{
Mitsuo Higaki
\thanks{
Department of Mathematics, 
Graduate School of Science, 
Kobe University, 
1-1 Rokkodai, Nada-ku, Kobe 657-8501, Japan.
\textit{E-mail address:}\texttt{higaki@math.kobe-u.ac.jp}
}
}

\date{}

\maketitle

\begin{abstract}
We establish the first quantitative Runge approximation theorem, with explicit $L^2$-estimates, for the 3d nonstationary Stokes system on a bounded spatial domain. This result addresses the two primary limitations of the qualitative result \cite{HigSue2025} obtained in collaboration with Franck Sueur: first, it bypasses the non-constructive Hahn-Banach theorem used in \cite{HigSue2025}, precluding quantitative estimates; and second, it extends the scope of the theory from interior approximations to the physically important initial-boundary value problem. Our proof is founded on the modern quantitative framework of R\"{u}land-Salo \cite{RulSal2019}, which we adapt to the Stokes system by combining semigroup theory with a quantitative approximation for the associated resolvent problem. 
\end{abstract}

\tableofcontents

    \section{Introduction}

The global approximation (or Runge-type) theorem is a foundational concept in partial differential equation (PDE) theory. It typically states that any local solution of an equation in a compact set $K$ can, under topological connectedness of the complement of $K$, be approximated by a global solution defined in the full space. This field has historically evolved in two distinct directions. 
The first is the qualitative direction, focused on proving the existence of such approximations. This theory, originating from Runge \cite{Run1885} and Mergelyan \cite{Mer1952} in complex analysis of one variable, is extended to elliptic PDEs by Lax \cite{Lax1956}, Malgrange \cite{Mal1955}, and Browder \cite{Bro1962}, and later generalized to non-elliptic equations such as the parabolic equations by Jones \cite{Jon1975} and Enciso--Garc\'{\i}a-Ferrero--Peralta-Salas \cite{EGFPS2019}; see also Dipierro-Savin-Valdinoci \cite{DSV2019} treating the fractional heat equations and Donaldson \cite{Don1993} the instanton equations. Furthermore, the Runge approximation property continues to play a crucial role in inverse problems, such as the Calder\'{o}n problem \cite{LLS2020}.

The second, more recent direction is the quantitative problem, which seeks to derive explicit estimates linking the approximation error to the given data; this is often called determining the cost of approximation \cite{DebKal2025}. This modern methodology was initiated by R\"{u}land-Salo \cite{RulSal2019}, updating the classical works \cite{Lax1956,Mal1955} in a quantitative manner. Indeed, quantitative approximation results have been established for various other operators, including non-elliptic ones by \cite{RulSal2020a,RulSal2020b,DebKal2025}, notably for the Schr\"{o}dinger operator by Enciso--Peralta-Salas \cite{EncPer2021}. The proof in this paper also builds upon these ideas of \cite{RulSal2019}.

The global approximation theorem for the 3d Stokes system is proved in \cite{HigSue2025}. The main result \cite[Theorem A]{HigSue2025} states that the velocity part of any local solution $(v,q)$ of the Stokes system in $K \subset \R^{4}_{+}$, a compact set satisfying the usual conditions, can be approximated by the velocity part of a solution $(u,p)$ of the Stokes system in $\R^{4}_{+}$. In general, these global solutions are unbounded in space \cite[Theorem B]{HigSue2025}. This behavior stands in sharp contrast to the heat equation \cite[Theorem 1.2]{EGFPS2019}, where approximations with decay at infinity are possible. These qualitative approximation theorems highlight a fundamental property of the underlying linear equations: \emph{rigidity}. This property, often manifested as a unique continuation principle, dictates that a local solution can, at best, only be approximated by a global one, excluding the possibility of a perfect match (i.e., $\eps=0$) unless the local solution happens to be the restriction of a global one. A very different scenario arises with the concept of \emph{flexibility}, a property observed in certain nonlinear systems, such as the 3d Euler system. For these ``flexible" equations, an entirely different theory applies in the context of weak solutions; see e.g. \cite{DeLelSze2009,LewPak2017} applying convex integration. Surprisingly, this framework allows for the construction of global in space-time weak solutions that exactly coincide with a given local smooth solution; see the recent remarkable progress \cite{EPTPS2025} for the 3d Euler system. The reader is also referred to \cite{DSV2019} for a ``flexible" Runge-type approximation to linear nonlocal equations and \cite{RulSal2020b} for the cost of approximation.

The results in \cite{HigSue2025} are inherently non-quantitative, meaning that the global approximations cannot be estimated by the given local data. This limitation is a direct consequence of the proof strategy employed. That proof consisted of two main stages: first, constructing an intermediate approximation $v_1$, under forcing with poles far from $K$, using the ``sweeping of poles and discretization" method \cite{Bro1962,EGFPS2019}; and second, extending $v_1|_{K}$ to a global solution $u$ using an ordinary differential equation (ODE) technique. The bottleneck, however, lies in the first stage. The ``sweeping" method, while effective for proving existence, relies fundamentally on a non-constructive Hahn-Banach theorem. It is this reliance that makes it impossible to derive quantitative estimates.

This situation mirrors the classic dichotomy observed by Bourgain-Brezis \cite{BouBre2003,BouBre2007} in the context of the divergence equation and Hodge systems: while duality arguments (Soft Analysis) provide existence of solutions in critical spaces, establishing explicit quantitative bounds requires a fundamentally different, constructive approach (Hard Analysis).

In the specific context of the Runge approximation, this transition from qualitative to quantitative is not merely a technical refinement but a crucial step for applications. As emphasized in recent works on the cost of approximation \cite{RulSal2019, DebKal2025}, explicit bounds on the global approximations are indispensable for applications to inverse problems and control theory, where the magnitude of the inputs directly determines the feasibility of the process.

Furthermore, the results in \cite{HigSue2025} are restricted to interior approximations, meaning that they do not cover the practically important case of initial-boundary value problems on a domain $D\subset\R^3$. This is a significant limitation, as many applications of the Stokes system involve fluid motion constrained by solid boundaries. While the interior framework of \cite{HigSue2025} is general enough to encompass pathological solutions, so-called parasitic solutions or Serrin's examples \cite{Ser1962}, it leaves open the problem of global approximation for solutions defined on a bounded spatial domain with prescribed boundary conditions.

This paper addresses these two limitations. We shift the focus from interior approximations of \cite{HigSue2025} to more physically relevant initial-boundary value problems on a bounded domain $D$. Our primary objective is to establish the first quantitative approximation theorems in this setting, specifically by deriving explicit $L^2$-estimates that link the global approximations directly to the initial data; see Theorems \ref{thm.main1} and \ref{thm.main2} below.

Let $D \subset \R^3$ be a bounded domain with smooth boundary. Set 
\[
    \begin{split}
    C^\infty_{0,\sigma}(D)
    &= \{\varphi\in C^\infty_0(D)^3~|~\opdiv\varphi=0\}, \\
    L^2_\sigma(D) 
    &= \big(\text{Closure of $C^\infty_{0,\sigma}(D)$ in $L^2(D)^3$}\big). 
    \end{split}
\]
Consider the 3d nonstationary Stokes system on $D$ under the no-slip condition: 
\begin{equation}\tag{S}\label{intro.eq.S}
    \left\{
    \begin{array}{ll}
    \partial_{t} v - \Delta v + \nabla q
    = 0&\mbox{in}\ D\times (0,\infty), \\
    \opdiv v
    = 0&\mbox{in}\ D\times [0,\infty), \\
    v
    = 0&\mbox{on}\ \partial D\times (0,\infty), \\
    v
    = v_0\in L^2_\sigma(D) &\mbox{on}\ D\times \{0\}.
    \end{array}\right.
\end{equation}
The solution of \eqref{intro.eq.S} can be represented by the Stokes semigroup, which we now recall. Let ${\mathbb P}: L^2(D)^3 \to L^2_\sigma(D)$ denote the orthogonal projection, called the Helmholtz projection, satisfying ${\mathbb P} \nabla p=0$ for $p\in H^1(D)$. Then we define 
\[
    {\mathbb A} = -{\mathbb P} \Delta,
    \qquad
    D({\mathbb A}) = L^{2}_{\sigma}(D) \cap H^{1}_0(D)^3 \cap H^{2}(D)^3. 
\]
The operator ${\mathbb A}$ is called the Stokes operator; see \cite[Section 2.1, I\hspace{-1.2pt}I\hspace{-1.2pt}I\hspace{-1.2pt}]{Soh2001book}. It is well known (see e.g. \cite[Chapter 2]{Lad1969book}) that ${\mathbb A}$ is nonnegative and self-adjoint in $L^{2}_{\sigma}(D)$ and that $-{\mathbb A}$ generates the $C_0$-analytic semigroup $\{e^{-t {\mathbb A}}\}_{t\ge0}$, called the Stokes semigroup. Moreover, ${\mathbb A}^{-1}$ is completely continuous, and the spectrum $\sigma({\mathbb A})$ is discrete. The elements $\lambda_1 \le \lambda_2 \le \cdots$ are positive and converge to $\infty$, and each element has finite multiplicity.

Our global approximation theorems for the Stokes system \eqref{intro.eq.S} are stated as follows.

\begin{theorem}\label{thm.main1}
Let $D \subset \R^3$ be a bounded domain whose complement $\R^3\setminus D$ is connected. Then, there exist positive constants $C,\nu$ for which the following statement holds: let $v_0\in D({\mathbb A})$ and set $v(t) = e^{-t {\mathbb A}} v_0$. For any $\eps>0$, there exists a smooth solution $(u,p)$ of 
\[
    \left\{
    \begin{array}{ll}
    \partial_{t} u - \Delta u + \nabla p
    = 0&\mbox{in}\ \R^3\times (0,\infty),\\
    \opdiv u
    = 0&\mbox{in}\ \R^3\times [0,\infty) 
    \end{array}\right.
\]
such that 
\begin{equation}\label{approx.thm.main1}
    \|v(t) - u(t)\|_{L^2(D)} 
    \le \eps \exp(t) 
    \|{\mathbb A} v_0\|_{L^2(D)},
    \quad
    t \ge 0, 
\end{equation}
and that
\begin{equation}\label{est.thm.main1}
\begin{split}
    |u(x,t)|
    &\le 
    \exp\big(\exp(C\eps^{-\nu})\big)
    \|{\mathbb A} v_0\|_{L^2(D)} \\
    &\quad
    \times
    \max\big\{
    \langle x \rangle^{\exp (C\eps^{-\nu})}, \,
    \exp(C\eps^{-1/2} |x|) 
    \big\}
    \exp(t), 
    \quad
    (x,t) \in \R^4_{+}. 
\end{split}
\end{equation}
\end{theorem}

\begin{theorem}\label{thm.main2}
Let $D$ be given as in Theorem \ref{thm.main1}. Then, for any $0<T<\infty$, there exist positive constants $C,\nu$ for which the following statement holds: let $v_0\in D({\mathbb A})$ and set $v(t) = e^{-t {\mathbb A}} v_0$. For any $\eps>0$, there exists a smooth solution $(u_{{\rm S}}, p_{{\rm S}})$ of 
\[
    \left\{
    \begin{array}{ll}
    \partial_{t} u_{{\rm S}} - \Delta u_{{\rm S}} + \nabla p_{{\rm S}}
    = 0&\mbox{in}\ \R^3\times (0,\infty),\\
    \opdiv u_{{\rm S}}
    = 0&\mbox{in}\ \R^3\times [0,\infty) 
    \end{array}\right.
\]
and initial data $u_0\in C^\infty_0(\R^3)$ such that, if one sets $u_{{\rm H}}(t) := e^{t\Delta} u_0$, 
\begin{equation}\label{approx.thm.main2}
    \|v(t) - u_{{\rm S}}(t) - u_{{\rm H}}(t)\|_{L^2(D)}
    \le
    \eps
    \|{\mathbb A} v_0\|_{L^2(D)},
    \quad
    0 \le t \le T, 
\end{equation}
and that 
\begin{align}
    |u_{{\rm S}}(x,t)|
    &\le 
    \exp\big(\exp(C\eps^{-\nu})\big)
    \|{\mathbb A} v_0\|_{L^2(D)} 
    \langle x \rangle^{\exp (C\eps^{-\nu})}
    \exp(t), 
    \quad
    (x,t) \in \R^4_{+}, \label{est.u1.thm.main2} \\
    |u_0(x)|
    &\le 
    \exp\big(C\eps^{-\nu}\exp(C\eps^{-\nu})\big)
    \|v_0\|_{L^2(D)}, 
    \quad
    x \in \R^3. \label{est.u0.thm.main2}     
\end{align}
\end{theorem}

\begin{remark}\label{rem.thm.main12}
\begin{enumerate}[(i)]
\item
One of the key motivations for this work was to improve the spatial behavior of the global approximations found in the previous qualitative work \cite{HigSue2025}. Indeed, the global approximations in \cite{HigSue2025}, constructed via a non-constructive Hahn-Banach theorem, exhibit exponential growth in space. It is worth noting that this spatial exponential growth in \cite{HigSue2025} could have been mitigated. In fact, the proof strategy in this paper, notably the decomposition of the resolvent solution into a Stokes part $u_1$ and a heat part $u_2$ in Proposition \ref{prop.global.approx.step3}, is conceptually separable from the quantitative idea. Had this decomposition been applied within the qualitative framework of \cite{HigSue2025}, it would have yielded a global approximation with only polynomial growth in space, akin to a qualitative counterpart of Theorem \ref{thm.main2}.

The true novelty of the present paper, therefore, lies not just in achieving polynomial spatial growth in Theorem \ref{thm.main2}, but in doing estimation quantitatively both in Theorems \ref{thm.main1} and \ref{thm.main2}. By replacing the abstract Hahn-Banach argument with a constructive proof based on the R\"{u}land-Salo framework and resolvent estimates, we establish explicit pointwise bounds for global approximations. This quantitative approach, however, introduces a trade-off in the temporal behavior. While the global approximations in \cite{HigSue2025} decay as $t\to\infty$ (see \cite[Section 3.6]{HigSue2025}), our current constructive proof via the Dunford integral necessarily results in approximations that grow exponentially in time as $\exp(t)$, as seen in \eqref{est.thm.main1} and \eqref{est.u1.thm.main2}.

\item
Contrary to the global-in-time approximation \eqref{approx.thm.main1} in Theorem \ref{thm.main1}, the approximation \eqref{approx.thm.main2} in Theorem \ref{thm.main2} is restricted to finite time intervals $[0,T]$. This restriction is due to the particular structure of the approximation, which decomposes the global solution into the Stokes part $u_{{\rm S}}$ and the heat semigroup part $u_{{\rm H}}$. Although the heat component $u_{{\rm H}}$, generated by the compactly supported initial data $u_0$, decays as $t\to\infty$, the Stokes component $u_{{\rm S}}$ still exhibits temporal growth, similar to the behavior of the full approximation $u$ in Theorem \ref{thm.main1}. Hence Theorem \ref{thm.main2} isolates a decaying component, but at the cost of restricting the approximation to finite time. This restriction is necessary as our proof does not guarantee the long-time behavior of this decomposition.

\item
The results of this paper are not restricted to the no-slip boundary condition. Our proof relies on two main aspects: the representation of the local solution $v(t)$ via the Dunford integral for the analytic semigroup $\{e^{-t {\mathbb A}}\}_{t\ge0}$ in Section \ref{sec.prf} and the quantitative approximation for resolvent problems in Section \ref{sec.resolvent}. The Dunford integral representation is valid for any sectorial operator $-\mathcal{A}$ whose resolvent set contains a sector with angle greater than $\pi/2$, generating the analytic semigroup $\{e^{-t \mathcal{A}}\}_{t\ge0}$. The core of this work, the quantitative approximation for the resolvent, is based on spectral properties and logarithmic stability in Lemma \ref{lem.est.stab}, which are not exclusive to the no-slip case. Consequently, any set of boundary conditions, such as the Navier-slip condition with favorable sign or periodic boundary conditions, that yields a Stokes operator $\mathcal{A}$ satisfying such abstract properties can be treated by the same quantitative methodology. The structure of the quantitative estimates in Theorems \ref{thm.main1} and \ref{thm.main2} should remain, though the constants $C$ involved in the estimates would depend on the given geometric and operator-theoretic properties of the chosen boundary condition. A detailed analysis of the cases where the Stokes operator $\mathcal{A}$ might possess a non-trivial kernel (e.g., under certain geometric symmetries or different types of slip conditions) is beyond the scope of this paper and will be addressed in future work.

\item
It would be of great interest to extend the $L^2$-quantitative global approximation results in Theorems \ref{thm.main1} and \ref{thm.main2} to an $L^p$-setting \cite{Sol1977,Gig1981,AbeGig2013}. However, such an extension presents significant challenges, as the R\"{u}land-Salo framework utilized in Section \ref{sec.resolvent} depends heavily on the underlying Hilbert space structure. 
\end{enumerate}
\end{remark}

The key to our proof lies in establishing a quantitative approximation $u$ for the homogeneous Stokes resolvent problem; see Proposition \ref{prop.global.approx}. A crucial step in this process, adapting the quantitative Runge approximation framework of R\"{u}land-Salo \cite{RulSal2019}, is the decomposition of the global approximation $u$ into two distinct parts (see Proposition \ref{prop.global.approx.step3}):
\begin{itemize}
\item
A solution $u_1$ to the Stokes resolvent problem, exhibiting polynomial growth in space

\item
A solution $u_2$ to the heat resolvent problem, exhibiting exponential growth in space
\end{itemize}
This decomposition allows us to derive precise, quantitative estimates for the growth of each component in terms of the approximation error $\eps$; see Propositions \ref{prop.global.approx.step4} and \ref{prop.global.approx.step5}, which are essential for estimating the constructed global approximations in Section \ref{sec.prf}.

The rest of this paper is organized as follows. In Section \ref{sec.resolvent}, we provide the global approximation theorem for the homogeneous Stokes resolvent problem. Section \ref{sec.prf} is devoted to the proofs of Theorem \ref{thm.main1} and Theorem \ref{thm.main2}, which are based on the results in Section \ref{sec.resolvent}.

    \paragraph{Notation}

In this paper, we will use the following notation.
\begin{itemize}
\item
We let ${\mathbb P}_{\R^3}$ denote the Helmholtz projection in $\R^3$ and ${\mathbb A}_{\R^3}$ the Stokes operator in $\R^3$.

\item
We let $B_R$ denote a ball in $\R^3$ centered at the origin with radius $R>0$.

\item
We let $(r,\theta,\phi)$ denote the spherical coordinates on $\R^{3}$
\[
    x
    =
    \left(
    \begin{matrix}
    x^{1}\\
    x^{2}\\
    x^{3}
    \end{matrix}
    \right)
    =
    \left(
    \begin{matrix}
    r\sin \theta \cos \phi\\
    r\sin \theta \sin \phi\\
    r\cos \theta
    \end{matrix}
    \right),
    \quad
    (r,\theta,\phi)
    \in
    (0,\infty)
    \times
    [0,\pi]
    \times
    [0,2\pi), 
\]
and $\bm{\hat r},\bm{\hat \theta},\bm{\hat \phi}$ the unit vectors
\[
    \bm{\hat r}
    =
    \left(
    \begin{matrix}
    \sin \theta \cos \phi\\
    \sin \theta \sin \phi\\
    \cos \theta
    \end{matrix}
    \right),
    \quad
    \bm{\hat \theta}
    =
    \left(
    \begin{matrix}
    \cos \theta \cos \phi\\
    \cos \theta \sin \phi\\
    -\sin \theta
    \end{matrix}
    \right),
    \quad
    \bm{\hat \phi}
    =
    \left(
    \begin{matrix}
    -\sin \phi\\
    \cos \phi\\
    0
    \end{matrix}
    \right).
\]
In addition, for a function $f$ on $\R^{3}$, we slightly abuse the notation to denote 
\[
f(x)
=f(r,\theta,\phi)
=f(r\sin\theta \cos\phi, r\sin\theta \sin\phi, r\cos\theta).
\]

\item
For $\delta\in(0,\pi)$, we set 
$
    \Sigma_\delta 
    = \{z\in\C\setminus\{0\}~|~|\oparg z|<\delta\}. 
$

\item 
For $Y\subset \R^3$ and a function $f$ defined on $Y$, we set 
\[
    \mathtt{e}_Y f 
    = \left\{
    \begin{array}{ll}
    f&\mbox{in}\ Y, \\
    0&\mbox{in}\ \R^3\setminus Y. 
    \end{array}\right. 
\]

\item 
For a domain $Y\subset \R^3$ and $f,g\in L^2(Y)^3$, we set 
$
    \langle f,g \rangle_Y
    = \int_Y f \cdot \overbar{g}. 
$

\item
We employ the following notation, where $C$ represents an implicit constant:
\begin{itemize}
\item
$A\lesssim B$ signifies that $A\le CB$ for some $C>0$.

\item
$A\approx B$ signifies that $C^{-1}A\le B\le CA$ for some $C\ge1$. 
\end{itemize}
\end{itemize}

    \section{Global approximation for the Stokes resolvent problem}
    \label{sec.resolvent}

The goal of this section is to prove the following proposition for global approximations to the local resolvent problem. The proof will be provided at the end of this section.

\begin{proposition}\label{prop.global.approx}
Let $D \subset \R^3$ be a bounded domain whose complement $\R^3\setminus D$ is connected. For $\lambda\in \Sigma_{\pi-\delta} \cap \{|\lambda| \ge 1\}$ with $\delta\in(0,\pi/2)$, let $v\in H^2(D)^3$ and $q\in H^1(D)$ satisfy
\[
    \left\{
    \begin{array}{ll}
    \lambda v - \Delta v + \nabla q
    = 0&\mbox{in}\ D, \\
    \opdiv v
    = 0&\mbox{in}\ D. 
    \end{array}\right.
\]
Then, for any $\eps>0$, there exists $(u,p)$ solving 
\[
    \left\{
    \begin{array}{ll}
    \lambda u - \Delta u + \nabla p
    = 0&\mbox{in}\ \R^3, \\
    \opdiv u
    = 0&\mbox{in}\ \R^3 
    \end{array}\right.
\]
such that $u$ approximates $v$ in $D$ as 
\[
    \|v - u\|_{L^2(D)} 
    \le 
    \eps 
    \Big( \|v\|_{H^{1}(D)} + \|q\|_{H^{1}(D)} \Big). 
\]
Moreover, $u$ can be decomposed into $u = u_1 + u_2$, where $u_1$ is a solution to the Stokes resolvent problem in $\R^3$ with polynomial growth, and $u_2$ is a solution to the heat resolvent problem in $\R^3$ with exponential growth; see Proposition \ref{prop.global.approx.step3}. Their quantitative estimates are provided in Propositions \ref{prop.global.approx.step4} and \ref{prop.global.approx.step5}, respectively.
\end{proposition}

We begin by stating a few lemmas. Let $\lambda\in \C\setminus\R_{\le0}$. Recall that the function
\[
    E_\lambda(x) = \frac{e^{-\sqrt{\lambda}|x|}}{4\pi|x|} 
    \quad \text{in} \mkern9mu \R^3 \setminus\{0\} 
\]
satisfies the distributional equation $\lambda u - \Delta u = \delta_0$ in $\R^3$. Set 
\[
    \vec{E}_\lambda = (E_\lambda,E_\lambda,E_\lambda). 
\]
For given $f\in L^2(\R^3)^3$, we define 
\[
    u = \vec{E}_\lambda \ast {\mathbb P}_{\R^3} f, 
    \qquad
    \nabla p = ({\rm Id} - {\mathbb P}_{\R^3}) f. 
\]
Then $(u,p)$ solves 
\[
    \left\{
    \begin{array}{ll}
    \lambda u - \Delta u + \nabla p
    = f&\mbox{in}\ \R^3, \\
    \opdiv u
    = 0&\mbox{in}\ \R^3.
    \end{array}\right.
\]

The following lemma is well-known, and thus we state it without proof.

\begin{lemma}\label{lem.est.R3}
Let $D\subset \R^3$ be a bounded domain and $f\in L^2(D)$. Set $w = E_\lambda \ast f$. Then, for any bounded domain $B \subset \R^3$, we have 
\[
    |\lambda| \|w\|_{L^2(B)}
    + |\lambda|^{1/2} \|\nabla w\|_{L^2(B)}
    +  \|\nabla^2 w\|_{L^2(B)}
    \lesssim 
    \|f\|_{L^2(D)}.
\]
The implicit constant is independent of $\lambda$.
\end{lemma}

A key lemma in the proof of Proposition \ref{prop.global.approx} is the following.

\begin{lemma}[logarithmic stability]\label{lem.est.stab}
Let $(v, q)$ be given as in Proposition \ref{prop.global.approx}. Let $D'$ be a subdomain of $D$ satisfying $\overbar{D'} \subset D$. Suppose that, for $0<\eta<\mathcal{E}$, 
\[
    \|v\|_{H^1(D)} \le \mathcal{E}, 
    \qquad
    \|v\|_{L^2(D')} \le \eta. 
\]
Then, there exist $C>1$ and $\mu>0$ independent of $\lambda, \eta, \mathcal{E}$ such that 
\begin{equation}\label{est.stab}
    \|v\|_{L^2(D)}
    \le 
    C \mathcal{E} 
    \bigg(\log \frac{\mathcal{E}}{\eta} \bigg)^{-\mu}. 
\end{equation}
\end{lemma}

\begin{proof}
Set
\[
    w(x,t) = e^{\lambda t} v(x),
    \qquad
    s(x,t) = e^{\lambda t} q(x). 
\]
Then $(w,s)$ satisfies the nonstationary Stokes system
\[
    \left\{
    \begin{array}{ll}
    \partial_t w - \Delta w + \nabla s
    = 0&\mbox{in}\ D \times\R, \\
    \opdiv w
    = 0&\mbox{in}\ D \times\R.
    \end{array}\right.
\]
By the 3-cylinder inequality for this system in \cite{LinWan2022} or derived from \cite{FabLeb2002}, 
\[
\begin{split}
    &\int_{-\tau/2}^{\tau/2} \int_{B_{R_2}} |w(x,t)|^2 \dd x \dd t \\
    &\le
    C \bigg(\int_{-\tau}^{\tau} \int_{B_{R_1}} |w(x,t)|^2 \dd x \dd t \bigg)^{\theta}
    \bigg(\int_{-\tau}^{\tau} \int_{B_{R_3}} |w(x,t)|^2 \dd x \dd t \bigg)^{1-\theta}, 
\end{split}
\]
for some $C>0$ and $0<\theta<1$. This in turn gives the 3-sphere estimate 
\[
    \|v\|_{L^2(B_{R_2})}
    \le 
    C \|v\|_{L^2(B_{R_1})}^{\theta}
    \|v\|_{L^2(B_{R_3})}^{1-\theta}. 
\]
Then an argument as in \cite{ARRV2009} yields the desired estimate \eqref{est.stab}. 
\end{proof}

Take $R>1$ sufficiently large so that $\overbar{D}\subset B_R$. Fix bounded domain $Y \subset\subset \R^3 \setminus \overline{B_R}$. Then we define the subspace of $L^2(D)^3$ by 
\[
    \mathcal{X} 
    = 
    \{v\in L^2(D)^3~|~{\mathbb P} (\lambda - \Delta)v = 0, \mkern9mu \opdiv v = 0 \}
\]
and the linear operator $T: L^2(Y)^3 \to \mathcal{X}$ by 
\[
    Tf 
    = 
    \vec{E}_{\lambda} \ast {\mathbb P}_{\R^3} (\mathtt{e}_Y f)|_D.
\]
The operator $T$ is well-defined. Indeed, $v:=Tf$ solves, for some $q\in H^1(D)$, 
\[
    \left\{
    \begin{array}{ll}
    \lambda v - \Delta v + \nabla q
    = 0&\mbox{in}\ D, \\
    \opdiv v
    = 0&\mbox{in}\ D.
    \end{array}\right.
\]
Let us denote by $\mathcal{K}$ the kernel of $T$ and by $\mathcal{K}^\perp$ the orthogonal space of $\mathcal{K}$ in $L^2(Y)^3$. Moreover, let us denote by $A$ the restriction of $T$ to $\mathcal{K}^\perp$.

\begin{lemma}\label{lem.A}
The following hold. 
\begin{enumerate}[(1)]
\item\label{item1.lem.A}
The adjoint $A^\ast: \mathcal{X} \to \mathcal{K}^\perp$ of $A$ is given by 
\[
    A^\ast v 
    = \vec{E}_{\bar{\lambda}} \ast {\mathbb P}_{\R^3} (\mathtt{e}_D v)|_Y.
\]

\item\label{item2.lem.A}
The mapping $A^\ast A$ is positive, compact, and self-adjoint on $\mathcal{K}^\perp$.

\item\label{item3.lem.A}
The range of $A$ is dense in $\mathcal{X}$.

\item\label{item4.lem.A}
There exist orthonormal bases 
\[
    \{f_j\}_{j=1}^{\infty} \subset \mathcal{K}^\perp, 
    \qquad
    \{v_j\}_{j=1}^{\infty} \subset \mathcal{X} 
\]
and positive constants $\{\alpha_j\}_{j=1}^{\infty}$ such that 
\[
    A f_j = \alpha_j v_j, 
    \qquad
    A^\ast v_j = \alpha_j f_j. 
\]
\end{enumerate}
\end{lemma}

\begin{proof}
We only prove (\ref{item1.lem.A}) as the others are standard. Indeed, (\ref{item2.lem.A}) and (\ref{item4.lem.A}) follow from standard functional analytic arguments, while (\ref{item3.lem.A}) is a consequence of the unique continuation property for the (adjoint) Stokes system; see \cite{FabLeb2002}. For $v\in L^2(D)$ and $f\in \mathcal{K}^\perp$, 
\[
    \langle Af, v \rangle_{L^2(D)}
    = \langle Tf, v \rangle_{L^2(D)}
    = \langle \vec{E}_{\lambda} \ast {\mathbb P}_{\R^3} (\mathtt{e}_Y f), \mathtt{e}_D v \rangle_{L^2(\R^3)}. 
\]
Then, we have 
\[
    \begin{split}
    \langle \vec{E}_{\lambda} \ast {\mathbb P}_{\R^3} (\mathtt{e}_Y f), \mathtt{e}_D v \rangle_{L^2(\R^3)} 
    &= \langle \vec{E}_{\lambda} \ast {\mathbb P}_{\R^3} (\mathtt{e}_Y f), {\mathbb P}_{\R^3} (\mathtt{e}_D v) \rangle_{L^2(\R^3)} \\
    &= \langle {\mathbb P}_{\R^3} (\mathtt{e}_Y f), \overbar{\vec{E}_{\lambda}} \ast {\mathbb P}_{\R^3} (\mathtt{e}_D v) \rangle_{L^2(\R^3)} \\
    &= \langle \mathtt{e}_Y f, \vec{E}_{\bar{\lambda}} \ast {\mathbb P}_{\R^3} (\mathtt{e}_D v) \rangle_{L^2(\R^3)} \\
    &= \langle f, \vec{E}_{\bar{\lambda}} \ast {\mathbb P}_{\R^3} (\mathtt{e}_D v)|_Y \rangle_{L^2(Y)}, 
    \end{split}
\]
which verifies the claim of (\ref{item1.lem.A}). 
\end{proof}

We now prove Proposition \ref{prop.global.approx}. The proof, inspired by the argument in \cite[Proof of Theorem 2.4]{EncPer2021}, consists of five steps. The first step is to show the following proposition.

\begin{proposition}\label{prop.global.approx.step1}
Let $(v, q)$ be given as in Proposition \ref{prop.global.approx}. Then, for $\eps > 0$, there exists $F\in L^2(Y)^3$ such that 
\[
    v_1 
    := 
    \vec{E}_{\lambda} \ast {\mathbb P}_{\R^3} (\mathtt{e}_Y F) 
\]
approximates $v$ in $D$ as 
\[
    \|v - v_1\|_{L^2(D)}
    \le
    \eps 
    \Big( \|v\|_{H^{1}(D)} + \|q\|_{H^{1}(D)} \Big). 
\]
Moreover, $F$ is quantitatively estimated as 
\begin{equation}\label{est.F}
    \|F\|_{L^2(Y)}
    \le
    \exp
    \bigg(
    C
    \Big(\frac{\langle \lambda \rangle}{\eps} \Big)^{4/\mu}
    \bigg)
    \|v\|_{L^2(D)}, 
\end{equation}
where $C$ is independent of $\eps, \lambda$ while $\mu$ is introduced in Lemma \ref{lem.est.stab}. 
\end{proposition}

\begin{proof}
Since $v\in \mathcal{X}$, one can expand $v$ by $\{v_j\}_{j=1}^{\infty}$ in Lemma \ref{lem.A} (\ref{item4.lem.A}): 
\[
    v = \sum_{j=1}^{\infty} \beta_j v_j. 
\]
For any $\alpha > 0$, we define 
\[
    F 
    = \sum_{\{j \,|\, \alpha_j>\alpha\}} 
    \frac{\beta_j}{\alpha_j} f_j
    \in \mathcal{K}^\perp. 
\]
Then we see that 
\begin{equation}\label{est0.F}
    \|F\|_{L^2(Y)} 
    = \bigg(\sum_{\{j \,|\, \alpha_j>\alpha\}} 
    \frac{|\beta_j|^2}{|\alpha_j|^2}\bigg)^{1/2}
    \le \frac{\|v\|_{L^2(D)}}{\alpha}.     
\end{equation}
Moreover, 
\[
    A F 
    = \sum_{\{j \,|\, \alpha_j>\alpha\}} \beta_j v_j 
    \in \mathcal{X} 
\]
approximates $v$ in the topology of $L^2(D)^3$. This approximation can be represented as 
\[
    A F
    = T F
    = \vec{E}_{\lambda} \ast {\mathbb P}_{\R^3} (\mathtt{e}_Y F)|_D. 
\]

Now we aim to estimate the error 
\[
    \mathsf{E} 
    := v - A F 
    = \sum_{\{j \,|\, \alpha_j \le \alpha\}} \beta_j v_j 
    \in \mathcal{X}. 
\]
Notice that
\[
    \langle \mathsf{E}, A F \rangle_{L^2(D)}
    = \langle v - A F, A F \rangle_{L^2(D)}
    = \|A F\|_{L^2(D)}^2 - \|A F\|_{L^2(D)}^2
    = 0.
\]
It trivially holds that
\begin{equation}\label{est0.mathsfE}
    \|\mathsf{E}\|_{L^2(D)} \le \|v\|_{L^2(D)}.     
\end{equation}
To improve this bound, we consider 
\[
    w 
    := \overbar{\vec{E}_{\lambda}} \ast {\mathbb P}_{\R^3} (\mathtt{e}_D \mathsf{E})
    = \vec{E}_{\bar{\lambda}} \ast {\mathbb P}_{\R^3} (\mathtt{e}_D \mathsf{E}). 
\]
By definition, $w$ and $\nabla s := ({\rm Id} - {\mathbb P}_{\R^3}) (\mathtt{e}_D \mathsf{E})\in L^2(\R^3)^3$ solve 
\[
    \left\{
    \begin{array}{ll}
    \bar{\lambda} w - \Delta w + \nabla s
    = \mathtt{e}_D \mathsf{E}&\mbox{in}\ \R^3, \\
    \opdiv w
    = 0&\mbox{in}\ \R^3. 
    \end{array}\right.
\]
Since
\[
    w|_Y 
    = \vec{E}_{\bar{\lambda}} \ast {\mathbb P}_{\R^3} (\mathtt{e}_D \mathsf{E})|_Y
    = A^\ast \mathsf{E}, 
\]
and
\[
    A^\ast \mathsf{E}
    = A^\ast \bigg(\sum_{\{j \,|\, \alpha_j \le \alpha\}} \beta_j v_j \bigg)
    = \sum_{\{j \,|\, \alpha_j \le \alpha\}} \alpha_j \beta_j f_j, 
\]
we have the smallness of $w$ localized in $Y$: 
\begin{equation}\label{est.w.smallness}
    \|w\|_{L^2(Y)} 
    = \|A^\ast \mathsf{E}\|_{L^2(Y)} 
    \le \alpha \|\mathsf{E}\|_{L^2(D)}. 
\end{equation}

The estimate \eqref{est0.mathsfE} can be improved by propagating \eqref{est.w.smallness} as follows. Observe that 
\[
\begin{split}
    \|\mathsf{E}\|_{L^2(D)}^2 
    &= \langle v - A F, \mathsf{E} \rangle_{L^2(D)} 
    = \langle v, \mathsf{E} \rangle_{L^2(D)} \\
    &= \langle \mathtt{e}_D v, \mathtt{e}_D \mathsf{E} \rangle_{L^2(\R^3)} 
    = \langle \mathtt{e}_D v, \bar{\lambda} w - \Delta w + \nabla s \rangle_{L^2(\R^3)}, 
\end{split}
\]
and that, by integration by parts, the right-hand side can be rewritten as 
\[
\begin{split}
    &\langle \mathtt{e}_D v, \bar{\lambda} w - \Delta w + \nabla s \rangle_{L^2(\R^3)} \\
    &= \int_D v\cdot 
    \overbar{(\bar{\lambda} w - \Delta w + \nabla s)} \\
    &= 
    - \int_D \nabla q\cdot \overbar{w}
    + \int_D v\cdot \nabla \overbar{s} 
    + \int_{\partial D} 
    (
    -v \cdot \partial_{\nu} \overbar{w} 
    + \partial_{\nu} v \cdot \overbar{w}
    ) 
    \dd \sigma \\
    &= 
    \int_{\partial D} 
    \big(
    - (q \nu) \cdot \overbar{w}
    + v \cdot (\overbar{s} \nu)
    -v \cdot \partial_{\nu} \overbar{w} 
    + \partial_{\nu} v \cdot \overbar{w}
    \big) 
    \dd \sigma \\
    &=: 
    I_1 + I_2 + I_3 + I_4, 
\end{split}
\]
where $\nu$ denotes the outward normal unit vector to $\partial D$. Let $B'$ be a ball sufficiently large so that $\overbar{D} \cup \overbar{Y} \subset B'$. Applying the trace theorem \cite[Chapter 7]{AdaFou2003book} and the interpolation
\begin{equation}\label{interp}
    \|w\|_{H^{1}(B'\setminus\overbar{D})}
    \lesssim
    \|w\|_{L^{2}(B'\setminus\overbar{D})}^{1/2}
    \|w\|_{H^{2}(B'\setminus\overbar{D})}^{1/2},    
\end{equation} 
we see that 
\begin{equation}\label{est.I1}
\begin{split}
    |I_1| 
    &\le 
    \|q\|_{L^{2}(\partial D)}
    \|w\|_{L^{2}(\partial D)} \\
    &\lesssim
    \|q\|_{L^{2}(D)}^{1/2}
    \|q\|_{H^{1}(D)}^{1/2}
    \|w\|_{L^{2}(B'\setminus\overbar{D})}^{1/2}
    \|w\|_{H^{1}(B'\setminus\overbar{D})}^{1/2} \\
    &\lesssim
    \|q\|_{L^{2}(D)}^{1/2}
    \|q\|_{H^{1}(D)}^{1/2}
    \|w\|_{L^{2}(B'\setminus\overbar{D})}^{3/4}
    \|w\|_{H^{2}(B'\setminus\overbar{D})}^{1/4}. 
\end{split}
\end{equation}

In a similar manner, we immediately have
\[
\begin{split}
    |I_2| 
    &\le
    \|v\|_{L^{2}(\partial D)}
    \|s\|_{L^{2}(\partial D)} \\
    &\lesssim
    \|v\|_{L^{2}(D)}^{1/2}
    \|v\|_{H^{1}(D)}^{1/2}
    \|s\|_{L^{2}(B'\setminus\overbar{D})}^{1/2}
    \|s\|_{H^{1}(B'\setminus\overbar{D})}^{1/2}. 
\end{split}
\]
If $\int_{B'\setminus\overbar{D}} s = 0$ is assumed without loss of generality, then 
\[
\begin{split}
    \|s\|_{L^{2}(B'\setminus\overbar{D})} 
    \lesssim
    \|\nabla s\|_{H^{-1}(B'\setminus\overbar{D})}. 
\end{split}
\]
Thus, from 
\[
\begin{split}
    \|\nabla s\|_{H^{-1}(B'\setminus\overbar{D})} 
    &= 
    \|\lambda w - \Delta w\|_{H^{-1}(B'\setminus\overbar{D})} \\
    &\lesssim
    \langle \lambda\rangle \|w\|_{L^2(B'\setminus\overbar{D})} 
    +  \|w\|_{H^1(B'\setminus\overbar{D})}, 
\end{split}
\]
by using the interpolation \eqref{interp}, we have
\[
\begin{split}
    \|s\|_{L^{2}(B'\setminus\overbar{D})}
    &\lesssim
    \langle \lambda \rangle \|w\|_{L^2(B'\setminus\overbar{D})}
    + \|w\|_{H^1(B'\setminus\overbar{D})} \\
    &\lesssim
    \|w\|_{L^2(B'\setminus\overbar{D})}^{1/2}
    \Big(
    \langle \lambda \rangle \|w\|_{L^2(B'\setminus\overbar{D})}^{1/2}
    + \|w\|_{H^{2}(B'\setminus\overbar{D})}^{1/2}
    \Big). 
\end{split}
\]
In addition, from $\nabla s = -(\bar{\lambda} w - \Delta w)$ in $B'\setminus\overbar{D}$, 
\[
\begin{split}
    \|\nabla s\|_{L^{2}(B'\setminus\overbar{D})}
    \lesssim
    \langle \lambda \rangle \|w\|_{L^2(B'\setminus\overbar{D})}
    +  \|w\|_{H^2(B'\setminus\overbar{D})}. 
\end{split}
\]
Hence we obtain 
\begin{equation}\label{est.I2}
\begin{split}
    |I_2| 
    &\lesssim
    \langle \lambda \rangle 
    \|v\|_{L^{2}(D)}^{1/2}
    \|v\|_{H^{1}(D)}^{1/2}
    \|w\|_{L^{2}(B'\setminus\overbar{D})}^{1/4} \\
    &\quad
    \times
    \Big(
    \|w\|_{L^2(B'\setminus\overbar{D})}^{1/2}
    + \|w\|_{H^{2}(B'\setminus\overbar{D})}^{1/2}
    \Big)^{1/2} \Big(
    \|w\|_{L^2(B'\setminus\overbar{D})} 
    + \|w\|_{H^{2}(B'\setminus\overbar{D})}
    \Big)^{1/2}.
\end{split}    
\end{equation}

By the trace theorem and \eqref{interp}, 
\begin{equation}\label{est.I3I4}
\begin{split}
    |I_3| + |I_4|
    &\le 
    \|v\|_{H^{1/2}(\partial D)} 
    \|\partial_{\nu} w\|_{H^{-1/2}(\partial D)}
    +  \|\partial_{\nu} v\|_{H^{-1/2}(\partial D)}
    \|w\|_{H^{1/2}(\partial D)} \\
    &\lesssim
    \|v\|_{H^{1}(D)} 
    \|w\|_{H^{1}(B'\setminus\overbar{D})} \\
    &\lesssim
    \|v\|_{H^{1}(D)} 
    \|w\|_{L^{2}(B'\setminus\overbar{D})}^{1/2} 
    \|w\|_{H^{2}(B'\setminus\overbar{D})}^{1/2}.     
\end{split}
\end{equation}
The estimates \eqref{est.I1}--\eqref{est.I3I4} combined with Lemma \ref{lem.est.R3} for $|\lambda| \ge 1$, show that 
\begin{equation}\label{est1.mathsfE}
\begin{split}
    \|\mathsf{E}\|_{L^2(D)}^2 
    &\lesssim
    \|q\|_{H^{1}(D)}
    \|\mathsf{E}\|_{L^2(D)}^{1/4}
    \|w\|_{L^{2}(B'\setminus\overbar{D})}^{3/4} \\
    &\quad
    + \langle \lambda \rangle 
    \|v\|_{H^{1}(D)} 
    \|\mathsf{E}\|_{L^2(D)}^{3/4}
    \|w\|_{L^{2}(B'\setminus\overbar{D})}^{1/4}. 
\end{split}
\end{equation}
To propagate the smallness of $w$ in \eqref{est.w.smallness} from $Y$ to $B'$, one can apply the stability estimate in Lemma \ref{lem.est.stab}. Observe that there exists $K>1$ independent of $\lambda$ such that 
\[
    \|w\|_{H^{1}(B'\setminus D)}
    \le
    K \|\mathsf{E}\|_{L^2(D)}, 
    \qquad
    \|w\|_{L^{2}(Y)}
    \le 
    K \alpha \|\mathsf{E}\|_{L^2(D)}. 
\]
Then, thanks to Lemma \ref{lem.est.stab}, for some $C>1$ and $\mu>0$ independent of $\lambda, \alpha$, 
\[
    \|w\|_{L^2(B'\setminus D)}
    \le 
    C \|\mathsf{E}\|_{L^2(D)}
    \bigg(
    \log \frac{1}{\alpha}
    \bigg)^{-\mu}. 
\]
If one substitutes this estimate into \eqref{est1.mathsfE}, then 
\[
    \|\mathsf{E}\|_{L^2(D)}
    \le
    C 
    \langle \lambda \rangle 
    \Big(
    \|v\|_{H^{1}(D)}
    + \|q\|_{H^{1}(D)}
    \Big)
    \bigg(
    \log \frac{1}{\alpha}
    \bigg)^{-\mu/4}. 
\]
Now define $\alpha$ as 
\[
\begin{split}
    \frac{1}{\alpha} 
    = 
    \exp
    \bigg(
    C^{4/\mu} 
    \Big(\frac{\langle \lambda \rangle}{\eps} \Big)^{4/\mu}
    \bigg). 
\end{split}
\]
Then we have 
\[
    \|\mathsf{E}\|_{L^2(D)}
    \le 
    \eps 
    \Big(
    \|v\|_{H^{1}(D)}
    + \|q\|_{H^{1}(D)}
    \Big), 
\]
which implies the approximation in the assertion. Moreover, $F$ is estimated as \eqref{est.F} from \eqref{est0.F} after the renaming of $C^{4/\mu}$ by $C$. This completes the proof. 
\end{proof}

In the second step, we show that $v_1$ in Proposition \ref{prop.global.approx.step1} can be approximated by a global approximation that solves the homogeneous Stokes resolvent problem in $\R^3$.

\begin{proposition}\label{prop.global.approx.step2}
Let $v_1$ be given as in Proposition \ref{prop.global.approx.step1}. Then, for $\eps>0$, there exists a smooth global approximation $u$ that solves 
\[
    \left\{
    \begin{array}{ll}
    \lambda u - \Delta u + \nabla p
    = 0&\mbox{in}\ \R^3, \\
    \opdiv u
    = 0&\mbox{in}\ \R^3, 
    \end{array}\right.
\]
for some smooth $p$, and approximates $v_1$ in $D$ as 
\[
    \|v_1 - u\|_{L^2(D)} 
    \le 
    \eps \|v\|_{L^2(D)}. 
\]
\end{proposition}

\begin{proof}
Recall that $Y \subset \R^3$ is chosen so that $Y \subset\subset \R^3 \setminus \overline{B_R}$ with ball $B_R$ satisfying $\overbar{D}\subset B_R$. Take $1<\rho<R$ and consider a smaller ball $B_{\rho} \subset B_{R}$.

Let $F \in L^2(Y)^3$ be the source term given by Proposition \ref{prop.global.approx.step1}, which defines the velocity field $v_1 := \vec{E}_{\lambda} \ast {\mathbb P}_{\R^3} (\mathtt{e}_Y F)$. By definition, $v_1$ and its associated pressure gradient $\nabla q_1 := ({\rm Id} - {\mathbb P}_{\R^3}) (\mathtt{e}_Y F)$ satisfy the vector Helmholtz equation with forcing
\[
    \left\{
    \begin{array}{ll}
    \lambda v_1 - \Delta v_1 
    = -\nabla q_1&\mbox{in}\ B_\rho, \\
    \opdiv v_1
    = 0&\mbox{in}\ B_\rho. 
    \end{array}\right.
\]
Our goal is to construct a global solution $u$ to the homogeneous Stokes resolvent problem that approximates $v_1$ in $D$. The strategy involves extending the restriction $v_1|_{B_\rho}$ to the whole space $\R^3$. This is achieved by using an explicit representation for solutions of the vector Helmholtz equation, obtained through an expansion into vector spherical harmonics \cite{BEG1985,VMK1988book}, and by truncating the series for $v_1$ at a sufficiently large frequency $l_0$. Since the explicit representation is detailed in \cite[Section 3.5]{HigSue2025}, we only recall the main formulas.

In the ball $B_\rho$, we expand $q_1$ in a series of (scalar) spherical harmonics 
\[
\begin{split}
    q_1(x,\lambda) 
    &= q_1(r,\theta,\phi,\lambda) 
    = \sum_{l=1}^{\infty} \sum_{m=-l}^{l} 
    b_{lm}(r,\lambda) Y_{lm} 
\end{split}
\]
and $v_1$ in a series of vector spherical harmonics 
\begin{equation}\label{expn.v1}
\begin{split}
    v_1(x,\lambda)
    &=
    v_1(r,\theta,\phi,\lambda)\\
    &=
    \sum_{l=1}^{\infty} \sum_{m=-l}^{l}
    \Big(
    c^{r}_{lm}(r,\lambda) \bm{Y}_{lm}
    + c^{(1)}_{lm}(r,\lambda) \bm{\Psi}_{lm}
    + c^{(2)}_{lm}(r,\lambda) \bm{\Phi}_{lm}
    \Big). 
\end{split}
\end{equation}
Here we took $b_{00}=0$ without loss of generality, and $c^{r}_{lm}, c^{(1)}_{lm}, c^{(2)}_{lm}$ are defined by 
\[
\begin{split}
    c^{r}_{lm}
    &=
    \langle v_1, \bm{Y}_{lm}\rangle_{S},\\
    c^{(1)}_{lm}
    &=
    \frac{1}{\mu_l}
    \langle v_1, \bm{\Psi}_{lm}\rangle_{S},
    \quad l\ge1, \\
    c^{(2)}_{lm}
    &=
    \frac{1}{\mu_l}
    \langle v_1, \bm{\Phi}_{lm}\rangle_{S},
    \quad l\ge1, 
\end{split}
\]
where $\langle f,g\rangle_{S} := \int_{S} f \cdot \overbar{g}$ with the unit sphere $S$ in $\R^3$ and $\mu_l:=l(l+1)$.

Let $b_{lm}, c^{r}_{lm}, c^{(1)}_{lm}, c^{(2)}_{lm}$ be regarded as functions of one variable $r$ in the following. From the calculations in \cite[Section 3.5]{HigSue2025}, we see that $b_{lm}$ is given by 
\begin{equation}\label{rep.b}
    b_{lm}(r) = B_{lm} r^l
\end{equation}
and $B_{lm}=B_{lm}(\lambda)$ by 
\begin{equation}\label{def.B}
    B_{lm}(\lambda)
    =
    \frac{l+2}{l} 
    \rho^{-l-2}
    \int_{0}^{\rho}
    \langle 
    (-\lambda v_1 + \Delta v_1)\cdot \bm{\hat r}, Y_{lm} 
    \rangle_{S}(r)
    r^2
    \dd r.
\end{equation}
Then $c^{r}_{lm}$ is given by
\begin{equation}\label{rep.cr}
    c^{r}_{lm}(r)
    = C^{r}_{lm} r^{-\frac32} I_{l+\frac12}(\sqrt{\lambda} r)
    + \Lambda_{lm}(r), 
\end{equation}
where $C^{r}_{lm}=C^{r}_{lm}(\lambda)$ is defined by 
\begin{equation}\label{def.Cr}
    C^{r}_{lm}(\lambda)
    =
    \frac{1}{\mathcal{I}_{l+\frac12,-1}(\sqrt{\lambda})}
    \int_{0}^{\rho}
    (c^{r}_{lm}(r) - \Lambda_{lm}(r))
    r^{-\frac32} \overbar{I_{l+\frac12}(\sqrt{\lambda} r)}
    r^{2} \dd r 
\end{equation}
with
\begin{equation}\label{def.I.l12}
    \mathcal{I}_{l+\frac12,j}(\sqrt{\lambda})
    =
    \int_{0}^{\rho} 
    r^{j} |I_{l+\frac12}(\sqrt{\lambda} r)|^{2}
    \dd r
\end{equation}
and $\Lambda_{lm}$ by
\begin{equation}\label{rep.Lambda}
\begin{split}
    \Lambda_{lm}(r)
    &= 
    -lB_{lm}
    \int_{0}^{r} 
    r^{-\frac32} K_{l+\frac12}(\sqrt{\lambda} r)
    s^{l+\frac32} I_{l+\frac12}(\sqrt{\lambda} s)
    \dd s\\
    &\quad
    - l B_{lm}
    \int_{r}^{\infty}
    r^{-\frac32} I_{l+\frac12}(\sqrt{\lambda} r)
    s^{l+\frac32} K_{l+\frac12}(\sqrt{\lambda} s) 
    \dd s.
\end{split}
\end{equation}
Also, $c^{(1)}_{lm}$ is given by
\begin{equation*}
    c^{(1)}_{lm}(r)
    =
    \frac{1}{\mu_{l} r}
    \frac{\dd}{\dd r} (r^{2}c^{r}_{lm})(r)
    =
    \frac{1}{\mu_{l}}
    \Big(
    2 c^{r}_{lm}(r)
    + r \frac{\dd c^{r}_{lm}}{\dd r}(r)
    \Big) 
\end{equation*}
and, more explicitly, by
\begin{equation}\label{rep.c1}
\begin{split}
    c^{(1)}_{lm}(r)
    &=
    \frac{C^{r}_{lm}}{\mu_l} 
    \Big( -l r^{-\frac32} I_{l+\frac12}(\sqrt{\lambda} r) + \sqrt{\lambda} r^{-\frac12} I_{l-\frac12}(\sqrt{\lambda} r) \Big) \\
    &\quad
    + \frac{1}{\mu_{l}}
    \Big(
    2 \Lambda_{lm}(r) 
    + r \frac{\dd \Lambda_{lm}}{\dd r}(r)
    \Big)
\end{split}
\end{equation}
with
\begin{equation}\label{rep.Lambda'}
\begin{split}
    \frac{\dd \Lambda_{lm}}{\dd r}(r)
    &= 
    -lB_{lm} 
    \bigg(
    \frac{\dd}{\dd r} \big(r^{-\frac32} K_{l+\frac12}(\sqrt{\lambda} r)\big) 
    \int_{0}^{r} s^{l+\frac32} I_{l+\frac12}(\sqrt{\lambda} s) \dd s \\
    &\qquad\qquad\quad
    + \frac{\dd}{\dd r} \big( r^{-\frac32} I_{l+\frac12}(\sqrt{\lambda} r) \big) 
    \int_{r}^{\infty} s^{l+\frac32} K_{l+\frac12}(\sqrt{\lambda} s) \dd s 
    \bigg) \\
    &=
    \sqrt{\lambda} lB_{lm} 
    \int_{0}^{r} r^{-\frac32} K_{l-\frac12}(\sqrt{\lambda} r) s^{l+\frac32} I_{l+\frac12}(\sqrt{\lambda} s) \dd s  \\
    & \quad 
    - \sqrt{\lambda} lB_{lm} 
    \int_{r}^{\infty} r^{-\frac32} I_{l-\frac12}(\sqrt{\lambda} r) s^{l+\frac32} K_{l+\frac12}(\sqrt{\lambda} s) \dd s 
    - \frac{l+2}{r} \Lambda_{lm}(r).
\end{split}
\end{equation}
Here the following well-known formulas are used: 
\[
\begin{split}
    \frac{\dd I_\nu}{\dd z}(z)
    =-\frac{\nu}{z} I_\nu(z) + I_{\nu-1}(z), 
    \qquad
    \frac{\dd K_\nu}{\dd z}(z) 
    =-\frac{\nu}{z} K_\nu(z) - K_{\nu-1}(z). 
\end{split}
\]
In addition, $c^{(2)}_{lm}$ is given by
\begin{equation}\label{rep.c2}
    c^{(2)}_{lm}(r)
    = C^{(2)}_{lm} r^{-\frac12} I_{l+\frac12}(\sqrt{\lambda} r), 
\end{equation}
where $C^{(2)}_{lm}=C^{(2)}_{lm}(\lambda)$ is defined by 
\begin{equation}\label{def.C2}
    C^{(2)}_{lm}(\lambda)
    =
    \frac{1}{\mathcal{I}_{l+\frac12,1}(\sqrt{\lambda})}
    \int_{0}^{\rho}
    \mu_l c^{(2)}_{lm}(r) 
    r^{-\frac12} \overbar{I_{l+\frac12}(\sqrt{\lambda} r)}
    r^{2} \dd r.
\end{equation}
We emphasize that all of $b_{lm}, c^{r}_{lm}, c^{(1)}_{lm}, c^{(2)}_{lm}$ above are defined for $r\in(0,\infty)$.

Next we construct the global approximation $u$ starting from \eqref{expn.v1}. Denote
\begin{equation}\label{def.norm}
    \|f\|
    = 
    \bigg(
    \int_{0}^{\rho}
    |f(r)|^2 r^2 \dd r
    \bigg)^{1/2}. 
\end{equation}
Then we have
\[
    \|v_1\|_{L^2(B_\rho)}^2 
    = 
    \sum_{l=1}^{\infty} \sum_{m=-l}^{l}
    \Big(
    \|c^{r}_{lm}\|^2 
    + \mu_l \|c^{(1)}_{lm}\|^2 
    + \mu_l \|c^{(2)}_{lm}\|^2 
    \Big). 
\]
Let $\Delta_S$ denote the Laplace-Beltrami operator on $S$. Since the relations 
\[
\begin{split}
    \Delta_S \bm{Y}_{lm} 
    &= -(\mu_l + 2) \bm{Y}_{lm} + 2 \mu_l \bm{\Psi}_{lm}, \\
    \Delta_S \bm{\Psi}_{lm} 
    &= 2\bm{Y}_{lm} - \mu_l \bm{\Psi}_{lm}, \\ 
    \Delta_S \bm{\Phi}_{lm} 
    &= -\mu_l \bm{\Phi}_{lm} 
\end{split}
\]
are equivalent to 
\[
\begin{split}
    \mu_l \bm{Y}_{lm} 
    &= -\frac{\mu_l}{\mu_l-2} \Delta_S \bm{Y}_{lm} - \frac{2\mu_l}{\mu_l-2} \Delta_S \bm{\Psi}_{lm}, \\
    \mu_l \bm{\Psi}_{lm} 
    &= -\frac{2}{\mu_l-2} \Delta_S \bm{Y}_{lm} - \frac{\mu_l+2}{\mu_l-2} \Delta_S \bm{\Psi}_{lm}, \\ 
    \mu_l \bm{\Phi}_{lm}
    &= -\Delta_S \bm{\Phi}_{lm}, 
\end{split}
\]
when $l\ge2$, we see from interpolation and integration by parts that 
\[
    \sum_{l=1}^{\infty} \sum_{m=-l}^{l} 
    \langle l \rangle^{4}
    \Big(
    \|c^{r}_{lm}\|^2 
    + \mu_l \|c^{(1)}_{lm}\|^2 
    + \mu_l \|c^{(2)}_{lm}\|^2 
    \Big) \\
    \lesssim 
    \|v_1\|_{H^{2}(B_\rho)}^2 
    \lesssim 
    \|F\|_{L^{2}(Y)}^2. 
\]
Therefore, a pointwise estimate of the series in $l$ 
\begin{equation}\label{est.c}
\begin{split}
    \sum_{m=-l}^{l} 
    \Big(
    \|c^{r}_{lm}\|^2 
    + \mu_l \|c^{(1)}_{lm}\|^2 
    + \mu_l \|c^{(2)}_{lm}\|^2 
    \Big) 
    \le 
    C \langle l \rangle^{-4} 
    \|F\|_{L^{2}(Y)}^2, 
    \quad
    l \ge 1 
\end{split}    
\end{equation}
holds with $C$ independent of $\eps, \lambda$, which in particular leads to, for any $l_0 \ge 1$,
\[
\begin{split}
    &\sum_{l\ge l_0}^{\infty} \sum_{m=-l}^{l} 
    \Big(
    \|c^{r}_{lm}\|^2 
    + \mu_l \|c^{(1)}_{lm}\|^2 
    + \mu_l \|c^{(2)}_{lm}\|^2 
    \Big) \\
    &\le
    C \|F\|_{L^{2}(Y)}^2 
    \sum_{l\ge l_0}^{\infty}
    \langle l \rangle^{-4} \\
    &\le
    C \|v\|_{L^2(D)}^2 
    \exp
    \bigg(
    2\Big(\frac{\langle \lambda \rangle}{\eps} \Big)^{4/\mu}
    \bigg)
    \langle l_0 \rangle^{-3}. 
\end{split}
\]
The estimate of $F$ in Proposition \ref{prop.global.approx.step1} is used in the last line. Thus, if $l_0$ is chosen so that 
\begin{equation}\label{def.l0}
\begin{split}
    l_0
    = 
    \bigg(
    \frac{\eps^2}{C+1}
    \bigg)^{-1/3}
    \exp
    \bigg(
    \frac23
    \Big(\frac{\langle \lambda \rangle}{\eps} \Big)^{4/\mu}
    \bigg), 
\end{split}
\end{equation}
then the globally defined velocity field 
\begin{equation}\label{def.u}
\begin{split}
    u(x,\lambda)
    &=
    u(r,\theta,\phi,\lambda)\\
    &:=
    \sum_{l=1}^{l_0} \sum_{m=-l}^{l}
    \Big(
    c^{r}_{lm}(r,\lambda) \bm{Y}_{lm}
    + c^{(1)}_{lm}(r,\lambda) \bm{\Psi}_{lm}
    + c^{(2)}_{lm}(r,\lambda) \bm{\Phi}_{lm}
    \Big)
\end{split}
\end{equation}
approximates $v_1$ in $D$ as 
$
    \|v_1 - u\|_{L^2(D)} 
    \le 
    \eps \|v\|_{L^2(D)}
$
This completes the proof. 
\end{proof}

In the third step, we detail the structure of the global approximation $u$ in Proposition \ref{prop.global.approx.step2}. This observation is not provided in the previous result \cite{HigSue2025}.

\begin{proposition}\label{prop.global.approx.step3}
Let $u$ be given as in Proposition \ref{prop.global.approx.step2}. Then, $u$ can be decomposed into 
\[
    u = u_1 + u_2, 
\]
where $u_1$ satisfies the Stokes resolvent problem, for some smooth $p_1$, 
\[
    \left\{
    \begin{array}{ll}
    \lambda u_1 - \Delta u_1 + \nabla p_1
    = 0&\mbox{in}\ \R^3, \\
    \opdiv u_1
    = 0&\mbox{in}\ \R^3 
    \end{array}\right.
\]
and $u_2$ the heat resolvent problem 
\[
    \left\{
    \begin{array}{ll}
    \lambda u_2 - \Delta u_2 
    = 0&\mbox{in}\ \R^3, \\
    \opdiv u_2
    = 0&\mbox{in}\ \R^3.
    \end{array}\right.
\]
Moreover, the following qualitative estimates hold: 
\begin{equation}\label{est.u1u2.qual}
    |u_1(x,\lambda)| 
    \lesssim
    \langle x \rangle^{l_0}, 
    \qquad
    |u_2(x,\lambda)| 
    \lesssim
    e^{(\Re\sqrt{\lambda}) |x|}, 
\end{equation}
where the implicit constants depend on $\eps, \lambda$ while $l_0$ is defined in \eqref{def.l0}.
\end{proposition}

\begin{proof}
Define 
\begin{equation}\label{def.c1lm}
\begin{split}
    c^{r,1}_{lm}(r,\lambda) 
    &= \Lambda_{lm}(r,\lambda), 
    \\
    c^{(1),1}_{lm}(r,\lambda) 
    &= \frac{1}{\mu_{l}} \Big( 2 \Lambda_{lm}(r,\lambda) + r \frac{\dd \Lambda_{lm}}{\dd r}(r,\lambda) \Big), 
    \\
    c^{(2),1}_{lm}(r,\lambda) 
    &= 0 
\end{split}
\end{equation}
and 
\begin{equation}\label{def.c2lm}
\begin{split}
    c^{r,2}_{lm}(r,\lambda) 
    &= C^{r}_{lm} r^{-\frac32} I_{l+\frac12}(\sqrt{\lambda} r), \\
    c^{(1),2}_{lm}(r,\lambda) 
    &= \frac{C^{r}_{lm}}{\mu_l} \Big( -l r^{-\frac32} I_{l+\frac12}(\sqrt{\lambda} r) + \sqrt{\lambda} r^{-\frac12} I_{l-\frac12}(\sqrt{\lambda} r) \Big), \\
    c^{(2),2}_{lm}(r,\lambda) 
    &= C^{(2)}_{lm} r^{-\frac12} I_{l+\frac12}(\sqrt{\lambda} r). 
\end{split}
\end{equation}
Then it is not hard to check that the following globally defined velocity fields 
\begin{equation}\label{def.u1u2}
\begin{split}
    u_1(x,\lambda)
    &=
    u_1(r,\theta,\phi,\lambda)\\
    &:=
    \sum_{l=1}^{l_0} \sum_{m=-l}^{l}
    \Big(
    c^{r,1}_{lm}(r,\lambda) \bm{Y}_{lm}
    + c^{(1),1}_{lm}(r,\lambda) \bm{\Psi}_{lm}
    + c^{(2),1}_{lm}(r,\lambda) \bm{\Phi}_{lm}
    \Big), \\
    u_2(x,\lambda)
    &=
    u_2(r,\theta,\phi,\lambda)\\
    &:=
    \sum_{l=1}^{l_0} \sum_{m=-l}^{l}
    \Big(
    c^{r,2}_{lm}(r,\lambda) \bm{Y}_{lm}
    + c^{(1),2}_{lm}(r,\lambda) \bm{\Psi}_{lm}
    + c^{(2),2}_{lm}(r,\lambda) \bm{\Phi}_{lm}
    \Big)
\end{split}
\end{equation}
satisfy the assertion by following the calculations in \cite[Section 3.5]{HigSue2025}. 
\end{proof}

In the fourth step, we quantify the estimate \eqref{est.u1u2.qual} for $u_1$ in Proposition \ref{prop.global.approx.step3}. Set
\begin{equation}\label{def.N}
    N_{\eps,\lambda}
    = 
    C
    \Big(\frac{\langle \lambda \rangle}{\eps} \Big)^{4/\mu},
    \quad
    C>0,
\end{equation}
where $\mu$ is introduced in Lemma \ref{lem.est.stab}. Then $l_0$ defined in \eqref{def.l0} is estimated as 
\begin{equation}\label{est.l0}
    l_0 
    \le
    \exp(N_{\eps,\lambda})
\end{equation}
if $C$ in \eqref{def.N} is chosen to be sufficiently large.

\begin{proposition}\label{prop.global.approx.step4}
Let $u_1$ be given as in Proposition \ref{prop.global.approx.step3}. Let $N_{\eps, \lambda}$ be defined in \eqref{def.N}. Then, $u_1$ is quantitatively estimated as 
\begin{equation}\label{est.u1}
\begin{split}
    |u_1(x,\lambda)|
    &\le
    \exp\big(\exp(N_{\eps,\lambda})\big)
    \|v\|_{L^2(D)} 
    \langle x \rangle^{\exp (N_{\eps,\lambda})}, 
\end{split}
\end{equation}
for sufficiently large $C$ independent of $\eps, \lambda$. 
\end{proposition}

\begin{proof}
All constants including implicit ones are independent of $\eps, \lambda, l$ in this proof. First, to estimate $u_1$ defined in \eqref{def.u1u2}, we consider $B_{lm}$ in \eqref{def.B}. By Lemma \ref{lem.est.R3}, we have 
\begin{equation}\label{est.B}
\begin{split}
    |B_{lm}(\lambda)|
    &\le 
    \langle \lambda \rangle \|v_1\|_{L^2(B_{\rho})}
    + \|\Delta v_1\|_{L^{2}(B_{\rho})}
    \lesssim
    \|F\|_{L^2(Y)}. 
\end{split}
\end{equation}

Second we consider the coefficients $c^{r,1}_{lm}, c^{(1),1}_{lm}, c^{(2),1}_{lm}$ in \eqref{def.c1lm}. We apply the quantitative estimates of the modified Bessel functions $K_\nu(z), I_\nu(z)$ for $\nu \ge 1/2$ and $z\in\C$ with $\Re z>0$. The estimate of $K_\nu(z)$ is as follows: by \cite[\href{http://dlmf.nist.gov/10.32.E9}{(10.32.9)}]{NIST:DLMF}, one can bound 
\[
    |K_\nu(z)| 
    =
    \bigg|
    \int_0^\infty e^{-z \cosh t} \cosh(\nu t) \dd t
    \bigg|
    \le
    K_\nu(\Re z).
\]
When $\nu=l+1/2$, by \cite[\href{http://dlmf.nist.gov/10.47.E9}{(10.47.9)}]{NIST:DLMF} and \cite[\href{http://dlmf.nist.gov/10.49.E12}{(10.49.12)}]{NIST:DLMF}, one can further bound
\[
    K_{l+\frac12}(\Re z)
    \le
    \sqrt{\frac{\pi}{2\Re z}} e^{-\Re z} 
    \sum_{k=0}^{l}
    \frac{(l+k)!}{2^k k! (l-k)!} (\Re z)^{-k}. 
\]
Thus we have
\begin{equation}\label{est.K}
\begin{split}
    &\big|K_{l+\frac12}(\sqrt{\lambda} r)\big| 
    \le 
    K_{l+\frac12}\big((\Re\sqrt{\lambda}) r\big) \\
    &\lesssim
    \left\{
    \begin{array}{ll}
    \displaystyle{
    \big((\Re\sqrt{\lambda}) r\big)^{-l-\frac12}
    \sqrt{\frac{\pi}{2}} 
    e^{-1}
    \sum_{k=0}^{l}
    \frac{(l+k)!}{2^k k! (l-k)!} 
    }
    &\mbox{if}\
    0 < r \le (\Re\sqrt{\lambda})^{-1},\\[15pt]
    \displaystyle{
    \big((\Re\sqrt{\lambda}) r\big)^{-\frac12} 
    e^{-(\Re\sqrt{\lambda}) r} 
    \sqrt{\frac{\pi}{2}} 
    e^{-1}
    \sum_{k=0}^{l}
    \frac{(l+k)!}{2^k k! (l-k)!} 
    }
    &\mbox{if}\ 
    r > (\Re\sqrt{\lambda})^{-1} 
    \end{array}\right. \\
    &\lesssim
    \left\{
    \begin{array}{ll}
    K_{l+\frac12}(1) 
    (\Re\sqrt{\lambda})^{-l-\frac12} r^{-l-\frac12}
    &\mbox{if}\
    0 < r \le (\Re\sqrt{\lambda})^{-1},\\[5pt]
    K_{l+\frac12}(1)
    (\Re\sqrt{\lambda})^{-\frac12} r^{-\frac12}
    e^{-(\Re\sqrt{\lambda}) r}
    &\mbox{if}\ 
    r > (\Re\sqrt{\lambda})^{-1}.
    \end{array}\right.
\end{split}    
\end{equation}
On the other hand, the estimate of $I_\nu(z)$ is as follows: by \cite[\href{http://dlmf.nist.gov/10.32.E2}{(10.32.2)}]{NIST:DLMF}, one can bound 
\[
\begin{split}
    |I_\nu(z)| 
    =
    \bigg|
    \frac{1}{\pi^{1/2} \Gamma(\nu+\frac12)} 
    \Big(\frac{z}{2}\Big)^\nu
    \int_{-1}^1 (1-t^2)^{\nu-\frac12} e^{zt} \dd t
    \bigg|
    \le
    \Big|\frac{z}{\Re z}\Big|^{\nu}
    I_\nu(\Re z).
\end{split} 
\]
If $0 \le \Re z \le 1$, by the definition \cite[\href{http://dlmf.nist.gov/10.25.E2}{(10.25.2)}]{NIST:DLMF}, one can further bound $I_\nu(\Re z) \le I_\nu(1) (\Re z)^\nu$. If $\Re z > 1$, since the function $f(x) := x^{1/2} e^{-x} I_\nu(x)$ is known to be strictly increasing on $(0,\infty)$ by \cite[Subsection 2.1]{Bar2010}, using \cite[\href{http://dlmf.nist.gov/10.40.E1}{(10.40.1)}]{NIST:DLMF}, one can further bound 
\[
    I_\nu(\Re z) 
    \le 
    (\Re z)^{-\frac12} e^{\Re z} 
    \limsup_{x\to\infty}f(x) 
    = 
    \frac{1}{\sqrt{2\pi}} (\Re z)^{-\frac12} e^{\Re z}. 
\]
Combining these two cases with $c_1 \le |\frac{\sqrt{\lambda}}{\Re\sqrt{\lambda}}| \le c_2$ for some $c_1, c_2$ depending on $\delta$, we have
\begin{equation}\label{est.I}
\begin{split}
    \big|I_{l+\frac12}(\sqrt{\lambda} r)\big| 
    &\le 
    c_2^{l+\frac12} I_{l+\frac12}\big((\Re\sqrt{\lambda}) r\big) \\
    &\le
    \left\{
    \begin{array}{ll}
    C^{l+\frac12} 
    (\Re\sqrt{\lambda})^{l+\frac12} r^{l+\frac12}
    &\mbox{if}\
    0 \le r \le (\Re\sqrt{\lambda})^{-1},\\[5pt]
    C^{l+\frac12} (\Re\sqrt{\lambda})^{-\frac12} r^{-\frac12}
    e^{(\Re\sqrt{\lambda}) r}
    &\mbox{if}\ 
    r > (\Re\sqrt{\lambda})^{-1}, 
    \end{array}\right.
\end{split}
\end{equation}
for some $C>1$. Here it is used that $I_{\nu}(1)$ is uniformly bounded in $\nu>0$; see \cite[\href{http://dlmf.nist.gov/10.37}{\S 10.37}]{NIST:DLMF}.

We turn to $c^{r,1}_{lm}, c^{(1),1}_{lm}, c^{(2),1}_{lm}$. The estimate of $\Lambda_{lm}$ in \eqref{rep.Lambda} is as follows. From 
\[
\begin{split}
    &\big|
    r^{-\frac32} K_{l+\frac12}(\sqrt{\lambda} r)
    s^{l+\frac32} I_{l+\frac12}(\sqrt{\lambda} s)
    \big|\\
    &\le
    \left\{
    \begin{array}{ll}
    C^{l+\frac12} K_{l+\frac12}(1) 
    r^{-l-2}
    s^{2l+2}
    &\mbox{if}\
    0 \le s \le r \le (\Re\sqrt{\lambda})^{-1},\\[5pt]
    C^{l+\frac12} K_{l+\frac12}(1) 
    (\Re\sqrt{\lambda})^{l} 
    r^{-2} 
    e^{-\Re(\sqrt{\lambda})r} 
    s^{2l+2}
    &\mbox{if}\ 
    0 \le s \le (\Re\sqrt{\lambda})^{-1} \le r,\\[5pt]
    C^{l+\frac12} 
    K_{l+\frac12}(1) 
    (\Re\sqrt{\lambda})^{-1} 
    r^{-2} 
    s^{l+1} 
    e^{-(\Re\sqrt{\lambda})(r-s)}
    &\mbox{if}\ 
    (\Re\sqrt{\lambda})^{-1} \le s \le r,
    \end{array}\right.
\end{split}
\]
we see that
\begin{equation}\label{est.Lambda1}
\begin{split}
    &\int_{0}^{r}
    \big|
    r^{-\frac32} K_{l+\frac12}(\sqrt{\lambda} r)
    s^{l+\frac32} I_{l+\frac12}(\sqrt{\lambda} s)
    \big|
    \dd s \\
    &\le
    \left\{
    \begin{array}{ll}
    l^{-1} C^{l+\frac12} K_{l+\frac12}(1) I_{l+\frac12}(1) 
    r^{l+1}
    &\mbox{if}\
    0 < r \le (\Re\sqrt{\lambda})^{-1},\\[5pt]
    l^{-1} C^{l+\frac12} K_{l+\frac12}(1) r^{l+1} 
    &\mbox{if}\ 
    r \ge (\Re\sqrt{\lambda})^{-1}
    \end{array}\right.\\
    &\le
    l^{-1} C^{l+\frac12} K_{l+\frac12}(1) r^{l+1}.
\end{split}
\end{equation}
In addition, from
\[
\begin{split}
    &\big|
    r^{-\frac32} I_{l+\frac12}(\sqrt{\lambda} r)
    s^{l+\frac32} K_{l+\frac12}(\sqrt{\lambda} s) 
    \big|\\
    &\le
    \left\{
    \begin{array}{ll}
    C^{l+\frac12} K_{l+\frac12}(1) 
    r^{l-1}
    s
    &\mbox{if}\
    0 < r \le s \le (\Re\sqrt{\lambda})^{-1},\\[5pt]
    C^{l+\frac12} K_{l+\frac12}(1) 
    (\Re\sqrt{\lambda})^{l} 
    r^{l-1}
    s^{l+1}
    e^{-\Re(\sqrt{\lambda})s}
    &\mbox{if}\ 
    0 < r \le (\Re\sqrt{\lambda})^{-1} \le s,\\[5pt]
    C^{l+\frac12} 
    K_{l+\frac12}(1)
    (\Re\sqrt{\lambda})^{-1} 
    r^{-2}
    s^{l+1}
    e^{-(\Re\sqrt{\lambda})(s-r)}
    &\mbox{if}\ 
    (\Re\sqrt{\lambda})^{-1} \le r \le s 
    \end{array}\right. 
\end{split}
\]
and an inequality 
\[
\begin{split}
    \int_{r}^{\infty} 
    s^{l+1} e^{-(\Re\sqrt{\lambda})(s-r)} \dd s
    &\le 
    (\Re\sqrt{\lambda})^{-1} r^{l+1}
    \sum_{k=0}^{l+1} \binom{l+1}{k} \Gamma(k+1) \\
    &\le 
    e (l+1)!
    (\Re\sqrt{\lambda})^{-1} r^{l+1}, 
    \quad
    r \ge (\Re\sqrt{\lambda})^{-1}, 
\end{split}
\]
we see that 
\begin{equation}\label{est.Lambda2}
\begin{split}
    &\int_{r}^{\infty}
    \big|
    r^{-\frac32} I_{l+\frac12}(\sqrt{\lambda} r)
    s^{l+\frac32} K_{l+\frac12}(\sqrt{\lambda} s) 
    \big|
    \dd s\\
    &\le
    \left\{
    \begin{array}{ll}
    C^{l+\frac12} 
    \Big(
    (l+1)! + K_{l+\frac12}(1) 
    \Big)
    (\Re\sqrt{\lambda})^{-2} 
    r^{l-1} 
    &\mbox{if}\
    0 < r \le (\Re\sqrt{\lambda})^{-1},\\[5pt]
    C^{l+\frac12} K_{l+\frac12}(1) (l+1)! (\Re\sqrt{\lambda})^{-2} 
    r^{l-1}
    &\mbox{if}\ 
    r \ge (\Re\sqrt{\lambda})^{-1} 
    \end{array}\right.\\
    &\le
    C^{l+\frac12} 
    K_{l+\frac12}(1) (l+1)! 
    \max\big\{
    (\Re\sqrt{\lambda})^{-l-1},\,
    r^{l+1}
    \big\}. 
\end{split}
\end{equation}
Combining \eqref{est.Lambda1}--\eqref{est.Lambda2} with \eqref{est.B} as well as the asymptotic estimate \cite[\href{http://dlmf.nist.gov/10.41.E2}{(10.41.2)}]{NIST:DLMF} 
\[
    K_{\nu}(1)
    \approx 
    \sqrt{\frac{\pi}{2\nu}} \Big(\frac{2\nu}{e}\Big)^{\nu} 
\]
when $\nu\to\infty$ and the estimate thanks to Stirling’s formula 
\[
    (l+1)! 
    \approx 
    \sqrt{2\pi(l+1)} \Big(\frac{l+1}{e}\Big)^{l+1}, 
\]
we estimate $\Lambda_{lm}=\Lambda_{lm}(r,\lambda)$ in \eqref{rep.Lambda} as follows: for sufficiently large $C>1$, 
\begin{equation}\label{est.Lambda}
\begin{split}
    |\Lambda_{lm}(r,\lambda)| 
    &\le 
    C^{l+\frac12} K_{l+\frac12}(1) (l+1)!
    \|F\|_{L^2(Y)} 
    \max\big\{
    (\Re\sqrt{\lambda})^{-l-1},\,
    r^{l+1}
    \big\} \\
    &\le 
    (Cl)^{Cl} 
    \|F\|_{L^2(Y)} 
    \max\big\{
    (\Re\sqrt{\lambda})^{-l-1},\,
    r^{l+1}
    \big\}. 
\end{split}    
\end{equation}
In a similar manner, using the representation \eqref{rep.Lambda'}, we estimate 
\begin{equation}\label{est.Lambda'}
\begin{split}
    &\Big|r\frac{\dd \Lambda_{lm}}{\dd r}(r,\lambda)\Big| \\
    &\le
    C l |B_{lm}| 
    \int_{0}^{r} 
    (\Re\sqrt{\lambda}) 
    \big|
    r^{-\frac12} K_{l-\frac12}(\sqrt{\lambda} r) s^{l+\frac32} I_{l+\frac12}(\sqrt{\lambda} s)
    \big|
    \dd s  \\
    & \quad 
    + C l |B_{lm}| 
    \int_{r}^{\infty} 
    (\Re\sqrt{\lambda}) 
    \big|
    r^{-\frac12} I_{l-\frac12}(\sqrt{\lambda} r) s^{l+\frac32} K_{l+\frac12}(\sqrt{\lambda} s) 
    \big|
    \dd s \\
    &\quad
    + (l+2) |\Lambda_{lm}(r)| \\
    &\le
    C^{l+\frac12} 
    \big(K_{l+\frac12}(1) + K_{l-\frac12}(1)\big)
    (l+1)!
    \|F\|_{L^2(Y)} 
    \max\big\{
    (\Re\sqrt{\lambda})^{-l-1},\,
    r^{l+1}
    \big\} \\
    &\le
    (Cl)^{Cl} 
    \|F\|_{L^2(Y)} 
    \max\big\{
    (\Re\sqrt{\lambda})^{-l-1},\,
    r^{l+1}
    \big\}.  
\end{split}
\end{equation}
Combining \eqref{est.Lambda}--\eqref{est.Lambda'}, we estimate $c^{r,1}_{lm}, c^{(1),1}_{lm}, c^{(2),1}_{lm}$ in \eqref{def.c1lm} as 
\begin{equation}\label{est.c1lm}
\begin{split}
    &|c^{r,1}_{lm}(r,\lambda)|
    + |c^{(1),1}_{lm}(r,\lambda)|
    + |c^{(2),1}_{lm}(r,\lambda)| \\
    &\le
    (Cl)^{Cl} 
    \|F\|_{L^2(Y)} 
    \max\big\{
    (\Re\sqrt{\lambda})^{-l-1},\,
    r^{l+1}
    \big\}. 
\end{split}
\end{equation}
Then, using the estimates for $\bm{Y}_{lm}, \bm{\Psi}_{lm}, \bm{\Phi}_{lm}$ 
\begin{equation}\label{est.VSH}
    \|\bm{Y}_{lm}\|_{L^\infty(S)}
    \lesssim
    l^{\frac12}, 
    \qquad    
    \|\bm{\Psi}_{lm}\|_{L^\infty(S)}
    \lesssim
    l^{\frac32}, 
    \qquad
    \|\bm{\Phi}_{lm}\|_{L^\infty(S)}
    \lesssim
    l^{\frac32}, 
\end{equation}
we conclude from \eqref{def.u1u2} and \eqref{est.c1lm} that 
\[
\begin{split}
    |u_1(x,\lambda)|
    &\le
    C
    \sum_{l=1}^{l_0} \sum_{m=-l}^{l}
    l^{\frac32} 
    \Big(
    |c^{r,1}_{lm}(r,\lambda)| 
    + |c^{(1),1}_{lm}(r,\lambda)| 
    + |c^{(2),1}_{lm}(r,\lambda)| 
    \Big) \\
    &\le
    (Cl_0)^{Cl_0} 
    \|F\|_{L^2(Y)}
    \max\big\{
    1,\, (\Re\sqrt{\lambda})^{-l_0-1}
    \big\}
    (r+1)^{l_0+1}.
\end{split}    
\]
Hence the desired estimate \eqref{est.u1} follows from \eqref{est.F}, \eqref{est.l0}, and $\Re\sqrt{\lambda} \ge c_2^{-1}$. 
\end{proof}

In the fifth and last step, we quantify the estimate \eqref{est.u1u2.qual} for $u_2$ in Proposition \ref{prop.global.approx.step3}.

\begin{proposition}\label{prop.global.approx.step5}
Let $u_2$ be given as in Proposition \ref{prop.global.approx.step3}. Let $N_{\eps, \lambda}$ be defined in \eqref{def.N}. Then, $u_2$ is quantitatively estimated as 
\begin{equation}\label{est.u2}
\begin{split}
    |u_2(x,\lambda)|
    &\le 
    \exp\big(\exp(N_{\eps,\lambda})\big) 
    \|v\|_{L^2(D)} 
    e^{(\Re\sqrt{\lambda})|x|}, 
\end{split} 
\end{equation}
for sufficiently large $C$ independent of $\eps, \lambda$. 
\end{proposition}

\begin{proof}
All constants including implicit ones are independent of $\eps, \lambda, l$ in this proof. To estimate the coefficients $c^{r,2}_{lm}, c^{(1),2}_{lm}, c^{(2),2}_{lm}$ in \eqref{def.c2lm}, we consider $C^{r}_{lm}(\lambda)$ in \eqref{def.Cr} and $C^{(2)}_{lm}(\lambda)$ in \eqref{def.C2}. A similar argument as in \cite[Appendix A]{EncPer2021} gives the lower bound 
\begin{equation}
\begin{split}
    \mathcal{I}_{l+\frac12,j}(\sqrt{\lambda}) 
    &=
    \int_{0}^{\rho} 
    r^{j} |I_{l+\frac12}(\sqrt{\lambda} r)|^{2}
    \dd r \\
    &\ge
    \frac{C}{\langle \sqrt{\lambda} \rangle l^2} 
    \Big(
    \frac{C\min\{|\sqrt{\lambda}|,\, 1\}}{2l+1} 
    \Big)^{2l+1} 
    e^{C\Re\sqrt{\lambda}}, 
    \quad
    j=-1,1, 
\end{split}
\end{equation}
where $\mathcal{I}_{l+\frac12,j}$ is defined in \eqref{def.I.l12}. Then, by the H\"{o}lder inequality, the notation \eqref{def.norm}, and the estimates \eqref{est.c} and \eqref{est.Lambda}, we estimate $|C^{r}_{lm}(\lambda)| + |C^{(2)}_{lm}(\lambda)|$ as 
\begin{equation}\label{est.CrC2}
\begin{split}
    |C^{r}_{lm}(\lambda)| + |C^{(2)}_{lm}(\lambda)|
    &\le 
    C \frac{\|c^{r}_{lm}\| 
    + \|\Lambda_{lm}\|}{\mathcal{I}_{l+\frac12,-1}(\sqrt{\lambda})^{\frac12}}
    + C \frac{l^2 \|c^{r}_{lm}\|}{\mathcal{I}_{l+\frac12,1}(\sqrt{\lambda})^{\frac12}} \\
    &\le 
    (Cl)^{Cl} 
    \|F\|_{L^2(Y)} 
    \max\big\{
    1,\, (\Re\sqrt{\lambda})^{-3l-1}
    \big\} 
    e^{-C\Re\sqrt{\lambda}}. 
\end{split}    
\end{equation}
Using this estimate, quantitative estimates for $I_{l+\frac12}$ and $I_{l-\frac12}(\sqrt{\lambda} r)$ (for the ones for $I_{l+\frac12}$, see \eqref{est.I}. The proof of the ones for $I_{l-\frac12}(\sqrt{\lambda} r)$ is similar), and \eqref{est.VSH}, we obtain 
\[
\begin{split}
    |u_2(x,\lambda)|
    &\le
    C
    \sum_{l=1}^{l_0} \sum_{m=-l}^{l}
    l^{\frac32} 
    \Big(
    |c^{r,2}_{lm}(r,\lambda)| 
    + |c^{(1),2}_{lm}(r,\lambda)| 
    + |c^{(2),2}_{lm}(r,\lambda)| 
    \Big) \\
    &\le
    C
    \sum_{l=1}^{l_0} \sum_{m=-l}^{l}
    (Cl)^{Cl} 
    \|F\|_{L^2(Y)}
    \max\big\{
    1,\, (\Re\sqrt{\lambda})^{-3l-1}
    \big\}
    e^{-C\Re\sqrt{\lambda}} \\
    &\qquad\qquad\qquad
    \times (1 + \Re\sqrt{\lambda})^{l+\frac12}
    (r+1)^{l}
    e^{(\Re\sqrt{\lambda}) r}.
\end{split}    
\]
Hence the desired estimate \eqref{est.u2} follows from \eqref{est.F}, \eqref{est.l0}, and $\Re\sqrt{\lambda} \ge c_2^{-1}$. 
\end{proof}

\begin{proofx}{Proposition \ref{prop.global.approx}}
The assertion is a combination of Propositions \ref{prop.global.approx.step1}--\ref{prop.global.approx.step5}. 
\end{proofx}

    \section{Proof of theorems}
    \label{sec.prf}

This section lifts the quantitative results for the resolvent problem, the goal of Section \ref{sec.resolvent}, to prove the main theorems for the time-dependent problem \eqref{intro.eq.S}. The crucial link is the Dunford integral, which represents the local solution $v(t) = e^{-t {\mathbb A}} v_0$ as an integral of the local resolvent $(\lambda+{\mathbb A})^{-1} v_0$ along a curve $\gamma$ in the complex plane.

Our proof strategy is structured to make this lift quantitative. Instead of approximating $(\lambda+{\mathbb A})^{-1} v_0$ directly, we introduce the comparison in three steps:
\begin{enumerate}[Step 1.]
\item
We define $V(\lambda)$ in \eqref{def.V.prf.thm.main1} 
below as the difference between the local resolvent $(\lambda+{\mathbb A})^{-1} v_0$ and the global one $(\lambda + {\mathbb A}_{\R^3})^{-1} v_0$. The difference $V(\lambda)$ is critical because it satisfies the homogeneous Stokes resolvent problem in the domain $D$.

\item 
We can now apply our main resolvent approximation result, Proposition \ref{prop.global.approx}, the culmination of Section \ref{sec.resolvent}, to this local homogeneous solution $V(\lambda)$. This yields a global homogeneous solution $U(\lambda)$ that approximates $V(\lambda)$ quantitatively.

\item 
We define our global time-dependent approximation $u(t)$ by combining the integral of $U(\lambda)$ with the integral of the global resolvent part; see \eqref{def.u.prf.thm.main1} below. The main technical task is then to estimate the error introduced by this substitution and $u(t)$.
\end{enumerate}

\begin{proofx}{Theorem \ref{thm.main1}}
By density argument, it suffices to prove the statement when $v_0\in C^\infty_{0,\sigma}(D)$.
Assume that $0<\eps<1$ without loss of generality. The proof relies on the representation of $\{e^{-t {\mathbb A}}\}_{t\ge0}$ by the Dunford integral \cite[Section $\mathrm{I}\hspace{-1.2pt}\mathrm{I}$.4.a]{EngNag2000book}
\[
    v(t)
    =
    e^{-t {\mathbb A}} v_0
    =
    \frac{1}{2\pi \ii}
    \int_{\gamma}
    e^{t \lambda}(\lambda+{\mathbb A})^{-1} v_0
    \dd \lambda,
    \quad 
    t>0, 
\]
where $\gamma$ denotes an oriented (counterclockwise) curve in $\C$: 
\[
    \gamma 
    = \Big\{|\oparg z|=\frac{3}{4}\pi, \mkern9mu |z|\ge 1 \Big\} 
    \cup \Big\{|\oparg z| \le \frac{3}{4}\pi, \mkern9mu |z|=1 \Big\}. 
\]
Using the property $\mathtt{e}_D v_0 = v_0$, we set
\begin{equation}\label{def.u1.prf.thm.main1}
    u_1(t)
    =
    e^{-t {\mathbb A}_{\R^3}} v_0 
    =
    \frac{1}{2\pi \ii}
    \int_{\gamma}
    e^{t \lambda}(\lambda + {\mathbb A}_{\R^3})^{-1} v_0
    \dd \lambda,
    \quad 
    t>0 
\end{equation}
and decompose the right-hand side of 
\begin{equation}\label{eq.v-u1.prf.thm.main1}
    v(t) - u_1(t)
    =
    \frac{1}{2\pi \ii}
    \int_{\gamma}
    e^{t \lambda}
    \Big(
    (\lambda + {\mathbb A})^{-1} v_0
    - (\lambda + {\mathbb A}_{\R^3})^{-1} v_0
    \Big)
    \dd \lambda 
\end{equation}
into $w_1(t) + w_2(t)$ by setting
\begin{equation}\label{def.w1w2.prf.thm.main1}
\begin{split}
    w_1(t) 
    &:=
    \frac{1}{2\pi \ii}
    \int_{\gamma\,\cap \{|\lambda|>L\}}
    e^{t \lambda} 
    \Big(
    (\lambda + {\mathbb A})^{-1} v_0
    - (\lambda + {\mathbb A}_{\R^3})^{-1} v_0 
    \Big)
    \dd \lambda, \\
    w_2(t)
    &:=
    \frac{1}{2\pi \ii}
    \int_{\gamma\,\cap \{|\lambda|\le L\}}
    e^{t \lambda} 
    \Big(
    (\lambda + {\mathbb A})^{-1} v_0
    - (\lambda + {\mathbb A}_{\R^3})^{-1} v_0 
    \Big)
    \dd \lambda, 
\end{split}
\end{equation}
with the constant $L>1$ that will be determined later depending on $\eps$.

The first term $w_1(t)$ in \eqref{def.w1w2.prf.thm.main1} is estimated as follows. From 
\[
\begin{split}
    w_1(t) 
    &= \frac{1}{2\pi \ii}
    \int_{\gamma\,\cap \{|\lambda|>L\}}    
    e^{t \lambda} 
    \bigg(
    (\lambda + {\mathbb A})^{-1} 
    - \frac{1}{\lambda}
    \bigg)
    v_0
    \dd \lambda \\
    &\quad
    - \frac{1}{2\pi \ii}
    \int_{\gamma\,\cap \{|\lambda|>L\}}    
    e^{t \lambda} 
    \bigg(
    (\lambda + {\mathbb A}_{\R^3})^{-1} 
    - \frac{1}{\lambda}
    \bigg)
    v_0
    \dd \lambda
\end{split}
\]
combined with the resolvent identities
\[
\begin{split}
    &\lambda (\lambda + {\mathbb A})^{-1} v_0 - v_0
    = (\lambda + {\mathbb A})^{-1} {\mathbb A} v_0, \\
    &\lambda (\lambda + {\mathbb A}_{\R^3})^{-1} v_0 - v_0
    = (\lambda + {\mathbb A}_{\R^3})^{-1} {\mathbb A}_{\R^3} v_0, \\
\end{split}
\]
one can bound $w_1(t)$ in $D$ as, using standard estimates, 
\[
\begin{split}
    \|w_1(t)\|_{L^2(D)} 
    &\lesssim 
    \int_{\gamma\,\cap \{|\lambda|>L\}}    
    \frac{e^{t \Re \lambda}}{|\lambda|} 
    \|(\lambda + {\mathbb A})^{-1} {\mathbb A} v_0\|_{L^2(D)}
    \dd |\lambda| \\
    &\quad
    + \int_{\gamma\,\cap \{|\lambda|>L\}}    
    \frac{e^{t \Re \lambda}}{|\lambda|} 
    \|(\lambda + {\mathbb A}_{\R^3})^{-1} {\mathbb A}_{\R^3} v_0\|_{L^2(D)}
    \dd |\lambda| \\
    &\lesssim 
    \big(\|{\mathbb A} v_0\|_{L^2(D)} + \|{\mathbb A}_{\R^3} v_0\|_{L^2(D)}\big)
    \int_{\gamma\,\cap \{|\lambda|>L\}}    
    \frac{e^{t \Re \lambda}}{|\lambda|^2} 
    \dd |\lambda| \\
    &\lesssim 
    \|{\mathbb A} v_0\|_{L^2(D)}
    \frac{1}{L} \exp\Big(-\frac{t L}{\sqrt{2}}\Big),
    \quad 
    t \ge 0.
\end{split}
\]
Now choose $L = \eps^{-1}$. Then we have 
\begin{equation}\label{est.w1.prf.thm.main1}
    \|w_1(t)\|_{L^2(D)} 
    \lesssim 
    \eps 
    \|{\mathbb A} v_0\|_{L^2(D)} 
    \exp\Big(-\frac{t}{\eps \sqrt{2}}\Big) ,
    \quad 
    t \ge 0.
\end{equation}

Next we consider the second term $w_2(t)$ in \eqref{def.w1w2.prf.thm.main1}. Fix $\lambda\in \gamma\,\cap \{|\lambda|\le L\}$. We apply the results in Section \ref{sec.resolvent} for the Stokes resolvent problem. Define $V$ in $D$ by 
\begin{equation}\label{def.V.prf.thm.main1}
    V 
    = (\lambda + {\mathbb A})^{-1} v_0 
    - (\lambda + {\mathbb A}_{\R^3})^{-1} v_0. 
\end{equation}
Since $V_1 := (\lambda + {\mathbb A})^{-1} v_0$ solves, for some $Q_1\in H^1(D)$, 
\[
    \left\{
    \begin{array}{ll}
    \lambda V_1 - \Delta V_1 + \nabla Q_1
    = v_0&\mbox{in}\ D, \\
    \opdiv V_1
    = 0&\mbox{in}\ D, \\
    V_1
    = 0&\mbox{on}\ \partial D 
    \end{array}\right.
\]
and $V_2 := (\lambda + {\mathbb A}_{\R^3})^{-1} v_0$ solves, for some $Q_2\in H^1(\R^3)$, 
\[
    \left\{
    \begin{array}{ll}
    \lambda V_2 - \Delta V_2 + \nabla Q_2
    = v_0&\mbox{in}\ \R^3, \\
    \opdiv V_2
    = 0&\mbox{in}\ \R^3,
    \end{array}\right.
\]
we see that $V$ solves, for some $Q\in H^1(D)$, 
\[
    \left\{
    \begin{array}{ll}
    \lambda V - \Delta V + \nabla Q
    = 0&\mbox{in}\ D, \\
    \opdiv V
    = 0&\mbox{in}\ D. 
    \end{array}\right.
\]
Hence one can apply Proposition \ref{prop.global.approx} to $(V,Q)$ with $\eps$ replaced by $\eps^3$. Then, there exists a smooth global approximation $U$ solving, with some smooth $P$, 
\[
    \left\{
    \begin{array}{ll}
    \lambda U - \Delta U + \nabla P 
    = 0&\mbox{in}\ \R^3, \\
    \opdiv U 
    = 0&\mbox{in}\ \R^3, 
    \end{array}\right.
\]
that approximates $V$ in $D$ as 
\[
    \|V - U\|_{L^2(D)} 
    \le
    \eps^3 
    \Big( \|V\|_{H^{1}(D)} + \|Q\|_{H^{1}(D)} \Big).
\]
Observe from $\lambda \in \gamma$ that $\|V\|_{H^1(D)} + \|\nabla Q\|_{L^2(D)} \lesssim \|v_0\|_{L^2(D)}$ and
\[
\begin{split}
    \|Q\|_{L^2(D)} 
    \le
    \|\nabla Q\|_{H^{-1}(D)} 
    &= 
    \|\lambda V - \Delta V\|_{H^{-1}(D)} \\
    &\lesssim
    \langle \lambda \rangle \|V\|_{L^2(D)} 
    + \|V\|_{H^1(D)}, 
\end{split}
\]
where $\int_{D} Q = 0$ is assumed without loss of generality. Thus we have
\[
    \|V - U\|_{L^2(D)} 
    \lesssim
    \eps^3 \langle \lambda \rangle \|v_0\|_{L^2(D)}.
\]
Using the global approximation $U$, we define
\begin{equation}\label{def.u2.prf.thm.main1}
    u_2(t)
    =
    \frac{1}{2\pi \ii}
    \int_{\gamma\,\cap \{|\lambda|\le L\}}
    e^{t \lambda} U(\lambda)
    \dd \lambda,
    \quad 
    t \ge 0. 
\end{equation}
Since $L=\eps^{-1}$, we have 
\begin{equation}\label{est.w2.prf.thm.main1}
\begin{split}
    \|w_2(t) - u_2(t)\|_{L^2(D)}
    &\lesssim 
    \int_{\gamma\,\cap \{|\lambda|\le L\}}
    e^{t \Re \lambda} \|V - U\|_{L^2(D)} 
    \dd |\lambda| \\
    &\lesssim 
    \eps^3 \|v_0\|_{L^2(D)} 
    \int_{\gamma\,\cap \{|\lambda|\le L\}}
    \langle \lambda \rangle e^{t \Re \lambda} 
    \dd |\lambda| \\
    &\lesssim 
    \eps \|v_0\|_{L^2(D)} e^{t}, 
    \quad 
    t \ge 0. 
\end{split}
\end{equation}

Now we define a globally defined velocity field $u$ by
\begin{equation}\label{def.u.prf.thm.main1}
    u(x,t) = u_1(x,t) + u_2(x,t), 
    \quad 
    t \ge 0, 
\end{equation}
with $u_1$ defined in \eqref{def.u1.prf.thm.main1} and $u_2$ in \eqref{def.u2.prf.thm.main1}. By definition, $u$ is a smooth solution of the nonstationary Stokes system in $\R^3\times[0,\infty)$ with some associated smooth pressure $p$. Moreover, we see from \eqref{eq.v-u1.prf.thm.main1}, \eqref{def.w1w2.prf.thm.main1}, \eqref{est.w1.prf.thm.main1} and \eqref{est.w2.prf.thm.main1} that $u$ approximates $v$ in $D$ as 
\[
\begin{split}
    \|v(t) - u(t)\|_{L^2(D)}
    &\le 
    \|w_1(t)\|_{L^2(D)}
    + \|w_2(t) - u_2(t)\|_{L^2(D)} \\
    &\lesssim
    \eps \exp\Big(-\frac{t}{\eps \sqrt{2}}\Big) 
    \|{\mathbb A} v_0\|_{L^2(D)} 
    + \eps e^{t} \|v_0\|_{L^2(D)} \\
    &\lesssim
    \eps e^{t} \|{\mathbb A} v_0\|_{L^2(D)},
    \quad 
    t \ge 0. 
\end{split}
\]
Hence the desired approximation \eqref{approx.thm.main1} holds true.

It remains to provide the quantitative estimate \eqref{est.thm.main1} of $u$. Firstly we have 
\begin{equation}\label{est.u1.prf.thm.main1}
\begin{split}
    \|u_1(t)\|_{L^\infty(\R^3)} 
    &= 
    \|e^{-t {\mathbb A}_{\R^3}} v_0\|_{L^\infty(\R^3)} \\
    &\lesssim
    \|e^{-t {\mathbb A}_{\R^3}} v_0\|_{H^2(\R^3)} \\
    &\lesssim 
    \|v_0\|_{H^2(\R^3)} 
    = 
    \|v_0\|_{H^2(D)} \\
    &\lesssim 
    \|{\mathbb A} v_0\|_{L^2(D)}, 
    \quad
    t \ge 0. 
\end{split}
\end{equation}
By Proposition \ref{prop.global.approx.step3}, the global approximation $U(\lambda)$ in \eqref{def.u2.prf.thm.main1} is decomposed into the sum of $U_1(\lambda)$ and $U_2(\lambda)$, where $U_1(\lambda)$ solves the  Stokes resolvent problem in $\R^3$ and $U_2(\lambda)$ the heat resolvent problem in $\R^3$. According to this decomposition, we define
\begin{equation}\label{def.u21u22.prf.thm.main1}
\begin{split}
    u_{2,i}(x,t)
    &:=
    \frac{1}{2\pi \ii}
    \int_{\gamma\,\cap \{|\lambda|\le L\}}
    e^{t \lambda} U_i(x,\lambda)
    \dd \lambda,
    \quad 
    i = 1,2. 
\end{split}
\end{equation}
Then $u_{2,1}(t)$ is a solution of the nonstationary Stokes system in $\R^3$. Moreover, by the quantitative estimate of $U_1$ in Proposition \ref{prop.global.approx.step4}, using the following estimates 
\[
\begin{split}
    N_{\eps,\lambda} 
    \le
    C \eps^{-8/\mu}, 
    \qquad
    \|V\|_{L^2(D)}
    \le
    C \|v_0\|_{L^2(D)} 
\end{split}
\]
valid for $\lambda \in \gamma\,\cap \{|\lambda|\le L\}$, one can estimate $u_{2,1}(t)$ as 
\begin{equation}\label{est.u21.prf.thm.main1}
\begin{split}
    |u_{2,1}(x,t)|
    \le 
    \exp\big(\exp(C\eps^{-8/\mu})\big)
    \|v_0\|_{L^2(D)} 
    \langle x \rangle^{\exp (C\eps^{-8/\mu})}
    e^t, 
    \quad
    (x,t) \in \R^4_{+}. 
\end{split}
\end{equation}
In the last line, we have taken $C$ in \eqref{def.N} to be sufficiently large.

In addition, $u_{2,2}(t)$ is a solution of the heat equation in $\R^3$. In a similar manner as above, by the estimate of $U_2$ in Proposition \ref{prop.global.approx.step5}, one can estimate $u_{2,2}(t)$ as
\begin{equation}\label{est.u22.prf.thm.main1}
\begin{split}
    |u_{2,2}(x,t)|
    \le 
    \exp\big(\exp(C\eps^{-8/\mu})\big)
    \|v_0\|_{L^2(D)} 
    e^{C\eps^{-1/2} |x|}
    e^t, 
    \quad
    (x,t) \in \R^4_{+}. 
\end{split}
\end{equation}
Now we conclude \eqref{est.thm.main1} by \eqref{def.u.prf.thm.main1}--\eqref{est.u22.prf.thm.main1} and replacing $8/\mu$ by $\nu$. The proof is complete. 
\end{proofx}

\begin{proofx}{Theorem \ref{thm.main2}}
This proof uses the notation in the proof of Theorem \ref{thm.main1}. By density argument, it suffices to prove the statement when $v_0\in C^\infty_{0,\sigma}(D)$. Set
\begin{equation}\label{def.uS.prf.thm.main2}
    u_{{\rm S}}(x,t) = u_1(x,t) + u_{2,1}(x,t). 
\end{equation}
Then the estimate \eqref{est.u1.thm.main2} follows by the proof of Theorem \ref{thm.main1} above.

Next we aim to construct the initial data $u_0$ in Theorem \ref{thm.main2} from $u_{2,2}(x,t)$. Firstly, observe that $u_{2,2}(x,t)$ satisfies the following inequality 
\[
    |u_{2,2}(x,t)|
    \le
    M e^{a|x|^2},
    \quad
    (x,t) \in \R^3\times(0,T), 
\]
for some $M,a>0$ allowed to depend on each fixed $0<T<\infty$. Hence, thanks to the uniqueness result \cite[Section 7.1 (b)]{Joh1991book}, by taking the initial trace $f(x) := u_{2,2}(x,0)$ and using the Gauss kernel $G_t$, we see that $u_{2,2}$ can be uniquely represented as 
\[
    u_{2,2}(x,t) 
    = 
    (G_t \ast f)(x,t)
    = 
    \int_{\R^3} 
    G_t(x-y) f(y) \dd y, 
    \quad 
    (x,t) \in \R^4_{+}. 
\]
Taking cut-off $\chi\in C^\infty_0(\R^3)$ being equal to $1$ on $B_M$, we consider 
\[
    f_M(x) 
    := 
    \chi(x) f(x)
    = 
    \chi(x) u_{2,2}(x,0), 
\]
which is smooth and compactly supported by definition. Since $f$ can be written as 
\[
\begin{split}
    f(x)
    =
    u_{2,2}(x,0)
    =
    \frac{1}{2\pi \ii}
    \int_{\gamma\,\cap \{|\lambda|\le L\}}
    U_2(x,\lambda)
    \dd \lambda,
\end{split}
\]
and $U_2$ satisfies \eqref{est.u2}, we have 
\begin{equation}\label{est1.prf.thm.main2}
\begin{split}
    &|u_{2,2}(x,t) - e^{t\Delta} f_M| \\
    &= 
    \bigg|
    \int_{\R^3} 
    G_t(x-y) 
    \big(f(y) - f_M(y) \big)
    \dd y
    \bigg| \\
    &= 
    \bigg|\int_{\R^3} 
    G_t(x-y) 
    \bigg(
    \frac{1}{2\pi \ii}
    \int_{\gamma\,\cap \{|\lambda|\le L\}}
    U_2(y,\lambda) 
    \big(1-\chi(y)\big)
    \bigg)
    \dd y
    \bigg| \\
    &\le
    \exp\big(\exp(C\eps^{-8/\mu})\big) 
    \|V\|_{L^2(D)} 
    \int_{|y| \ge M} 
    G_t(x-y) 
    e^{C\eps^{-1/2} |y|} 
    \dd y. 
\end{split}
\end{equation}

For given $a>1$ and $0<T<\infty$, one can estimate the integral 
\[
    I_a(x,t)
    :=
    \int_{|y| \ge M} 
    G_t(x-y) 
    e^{a |y|} 
    \dd y, 
    \quad
    (x,t) \in D \times [0,T]
\]
by choosing $M$ sufficiently large but independently of $a$ as follows. Let $R>1$ satisfy $\overbar{D}\subset B_R$. Then, by taking $M$ large enough depending on $R,T$, we can obtain 
\[
    G_t(x-y)  
    \le 
    G_T(x-y)  
    \le 
    \frac{1}{(4\pi T)^{3/2}}
    \exp\Big(-\frac{(|y|-R)^2}{4T}\Big),
    \quad
    |y| \ge M. 
\]
Thus, in the spherical coordinates,  
\[
\begin{split}
    I_a(x,t) 
    &\le 
    \frac{1}{(4\pi T)^{3/2}}
    \int_{|y| \ge M} 
    \exp\Big(-\frac{(|y|-R)^2}{4T} + a|y|\Big) 
    \dd y \\
    &\le
    C_T
    \int_{M}^{\infty} 
    \exp\Big(-\frac{(r-R)^2}{4T} + ar\Big)
    r^2
    \dd r,  
\end{split}
\]
for some $C_T$ depending on $T$. From 
\[
\begin{split}
    -\frac{(r-R)^2}{4T} + ar
    &=
    -\frac{(r - r_0)^2}{4T} 
    + aR + a^2T, 
    \qquad
    r_0 := R + 2aT, 
\end{split}
\]
we see that 
\begin{equation}\label{est0.I.prf.thm.main2}
    I_a(x,t) 
    \le 
    C_T
    \exp\Big(a^2(R+T)\Big) 
    \int_{M}^{\infty} 
    \exp\Big(-\frac{(r - r_0)^2}{4T}\Big)
    r^2
    \dd r. 
\end{equation}
By change of variables and the Young inequality, 
\[
\begin{split}
    &\int_{M}^{\infty} 
    \exp\Big(-\frac{(r - r_0)^2}{4T}\Big)
    r^2
    \dd r \\
    &=
    \int_{M-r_0}^{\infty} 
    \exp\Big(-\frac{u^2}{4T}\Big)
    (u+r_0)^2
    \dd u \\
    &\le
    2\int_{M-r_0}^{\infty} 
    \exp\Big(-\frac{u^2}{4T}\Big)
    (u^2 + r_0^2)
    \dd u. 
\end{split}
\]
Moreover, since 
\[
    \exp\Big(-\frac{u^2}{4T}\Big)
    \le 
    \exp\Big(-\frac{(M - r_0)^2}{8T}\Big)
    \exp\Big(-\frac{u^2}{8T}\Big), 
    \quad
    u \ge M-r_0, 
\]
we have 
\[
\begin{split}
    &\int_{M}^{\infty} 
    \exp\Big(-\frac{(r - r_0)^2}{4T}\Big)
    r^2
    \dd r \\
    &\le
    2\exp\Big(-\frac{(M - r_0)^2}{8T}\Big) 
    \int_{M-r_0}^{\infty} 
    \exp\Big(-\frac{u^2}{8T}\Big)
    (u^2 + r_0^2)
    \dd u \\
    &\le
    C_T a^2
    \exp\Big(-\frac{(M - r_0)^2}{8T}\Big).
\end{split}
\]
Substituting this estimate into \eqref{est0.I.prf.thm.main2}, we obtain 
\[
    I_a(x,t) 
    \le 
    C_T
    a^2
    \exp\Big(-\frac{(M - r_0)^2}{8T} + a^2(R+T)\Big). 
\]
Hence, by choosing $M$ so that 
\begin{equation}\label{def.M.prf.thm.main2}
    M 
    = 
    - C_0 
    \big(\log (a^{-4})\big) a 
    \exp\big(C_0 a^{8/\mu}\big) 
    (R+T)^{1/2} T^{1/2} 
    + r_0, 
\end{equation}
and by taking $C_0$ sufficiently large depending on $R,T$, then we conclude that 
\begin{equation}\label{est.I.prf.thm.main2}
    I_a(x,t) 
    \le 
    C_1 a^{-2} \exp\big(-C_1 \exp(a^{8/\mu})\big), 
    \quad
    (x,t) \in D \times [0,T]
\end{equation}
for some $C_1$ independent of $a$ and depending on $T,R$.

With the estimate \eqref{est.I.prf.thm.main2} with $a = C\eps^{-1/2}$, we go back to \eqref{est1.prf.thm.main2} to see that 
\begin{equation}\label{est2.prf.thm.main2}
\begin{split}
    |u_{2,2}(x,t) - e^{t\Delta} f_M| 
    \le
    \eps
    \|v_0\|_{L^2(D)}, 
    \quad 
    (x,t) \in D \times [0,T], 
\end{split}
\end{equation}
by taking $C_0$ in \eqref{def.M.prf.thm.main2} to be sufficiently large again if needed. Fixing such $M$, we set 
\begin{equation}\label{def.u0.prf.thm.main2}
    u_0(x) 
    = f_M(x) 
    = \chi(x) u_{2,2}(x,0) 
    = \chi(x) \bigg(\frac{1}{2\pi \ii}
    \int_{\gamma\,\cap \{|\lambda|\le L\}}
    U_2(x,\lambda)
    \dd \lambda
    \bigg). 
\end{equation}
Then the desired approximation \eqref{approx.thm.main2} follows from the definitions of $u_{{\rm S}}$ in \eqref{def.uS.prf.thm.main2} and of $u_{{\rm H}}$ by $u_{{\rm H}}:=e^{t\Delta}u_0$, the estimates in the proof of Theorem \ref{thm.main1} above, and \eqref{est2.prf.thm.main2}. Moreover, by setting $\nu = \max\{8/\mu, 1/2\}$, we see that the estimate \eqref{est.u0.thm.main2} for the initial data $u_0$ is a consequence of \eqref{est.u22.prf.thm.main1}, \eqref{def.M.prf.thm.main2} and \eqref{def.u0.prf.thm.main2}. This completes the proof of Theorem \ref{thm.main2}. 
\end{proofx}

    \section*{Acknowledgment}

The author is grateful to Franck Sueur for fruitful discussions and valuable comments, which significantly improved the quality of this manuscript. This work was supported by JSPS KAKENHI Grant Numbers JP 25K17278 and JP 25K00915.

\end{document}